\documentclass[letterpaper,10pt]{amsart}
\usepackage{amsmath,amsthm,amstext}
\usepackage{amsfonts,amssymb}
\usepackage{color}
\usepackage{bm}
\usepackage[mathscr]{eucal}
\usepackage[]{hyperref}
\usepackage{setspace}
\usepackage{tikz}

\newtheorem{theorem}{Theorem}[section]
\newtheorem{lemma}[theorem]{Lemma}
\newtheorem{proposition}[theorem]{Proposition}
\newtheorem{corollary}[theorem]{Corollary}
\newtheorem{conjecture}[theorem]{Conjecture}
\theoremstyle{definition}
\newtheorem{remark}[theorem]{Remark}
\newtheorem{definition}[theorem]{Definition}

\numberwithin{equation}{section}

\renewcommand{\l}{\lambda}

\newcommand{\RR}{\ensuremath{\mathbb{R}}}

\newcommand{\sph}{\ensuremath{\mathbb{S}}}
\newcommand{\prtl}{\ensuremath{\partial}}
\newcommand{\p}{\ensuremath{\partial}}
\DeclareMathOperator{\supp}{supp}
\newcommand{\g}{\ensuremath{\mathrm{g}}}
\newcommand{\h}{\ensuremath{\mathrm{h}}}
\newcommand{\To}{\longrightarrow}
\newcommand{\upd}{\ensuremath{\mathrm{d}}}
\newcommand{\defeq}{\ensuremath{\stackrel{\text{def}}{=}}}

\DeclareMathOperator{\Dom}{Dom}
\newcommand{\veps}{\ensuremath{\varepsilon}}
\newcommand{\bbR}{\ensuremath{\mathbb{R}}} 
\newcommand{\bbC}{\ensuremath{\mathbb{C}}} 
\newcommand{\bbZ}{\ensuremath{\mathbb{Z}}} 
\newcommand{\bbS}{\ensuremath{\mathbb{S}}} 
\newcommand{\calU}{\ensuremath{\mathcal{U}}} 
\newcommand{\calUdot}{\ensuremath{\dot{\mathcal{U}}}} 
\newcommand{\cone}{\ensuremath{C(\mathbb{S}^1_\rho)}}
\newcommand{\geom}{\ensuremath{\text{\textup{geom}}}} 
\newcommand{\diff}{\ensuremath{\text{\textup{diff}}}} 
\DeclareMathOperator{\Id}{Id} 
\newcommand{\del}{\ensuremath{\partial}} 
\newcommand{\Dir}{\ensuremath{\text{Dir}}} 
\newcommand{\Neu}{\ensuremath{\text{Neu}}} 

\title[Strichartz estimates on flat cones]{Strichartz estimates for the wave
  equation on flat cones}

\author[M. D. Blair]{Matthew D. Blair}
\address{Department of Mathematics and Statistics, University of New Mexico,
  Albuquerque, NM 87131, USA}
\email{blair@math.unm.edu}

\author[G. A. Ford]{G. Austin Ford}
\address{Department of Mathematics, Northwestern University, Evanston, IL 60208,
  USA}
\email{aford@math.northwestern.edu}

\author[J. L. Marzuola]{Jeremy L. Marzuola}
\address{Department of Mathematics, University of North Carolina -- Chapel Hill,
  Chapel Hill, NC 27599, USA}
\email{marzuola@math.unc.edu}

\begin{document}

\begin{abstract}
  We consider the solution operator for the wave equation on the flat Euclidean
  cone over the circle of radius $\rho > 0$, the manifold $\mathbb{R}_+ \times
  \left(\mathbb{R} \big/ 2\pi\rho \mathbb{Z}\right)$ equipped with the metric
  $\g(r,\theta) = \upd r^2 + r^2 \, \upd\theta^2$.  Using explicit
  representations of the solution operator in regions related to flat wave
  propagation and diffraction by the cone point, we prove dispersive estimates
  and hence scale invariant Strichartz estimates for the wave equation on flat
  cones.  We then show that this yields corresponding inequalities on wedge
  domains, polygons, and Euclidean surfaces with conic singularities.  This in
  turn yields well-posedness results for the nonlinear wave equation on such
  manifolds.  Morawetz estimates on the cone are also treated.
\end{abstract}

\maketitle


\section{Introduction}
\label{sec:Intro}
Let $\cone$ denote the flat cone over the circle of radius $\rho>0$, defined as
the product manifold $\cone=\RR_+ \times\left(\mathbb{R} \big/ 2\pi\rho
  \mathbb{Z}\right)$ equipped with the (incomplete) metric $\g(r,\theta) = \upd
r^2 + r^2 \, \upd\theta^2$.  In this work, we consider solutions $u:\RR \times
\cone \To \bbC$ to the initial value problem for the wave equation on $\cone$,
\begin{equation}
  \label{eqn:wave}
  \left\{
    \begin{aligned}
      \left(D_t^2 - \Delta_\g \right) u(t,r,\theta) &= 0 \\
      u(0,r,\theta) &= f(r,\theta) \\
      \prtl_t u (0,r,\theta) &= g(r,\theta) . \\
    \end{aligned}
  \right.
\end{equation}
Here, we use $\Delta_\g$ to denote the Friedrichs extension of the
Laplace-Beltrami operator acting on
$\mathcal{C}^\infty_c\!\left(C(\bbS^1_\rho)\right)$, and we write $D_t =
\frac{1}{i} \, \prtl_t$ for the Fourier-normalized time derivative.

Solutions to the wave equation on cones and related spaces have been extensively
studied over the years, beginning with the seminal work of Sommerfeld
\cite{Som}.  In the setting of metric cones, Cheeger and Taylor
\cite{CT82,CT82no2} established the propagation of singularities and provided
explicit formulae for solutions in terms of the functional calculus of the
Laplace-Beltrami operator on the cross-section.  This was then expanded upon by
Melrose and Wunsch \cite{melrosewunsch}, who proved a propagation of
singularities theorem for solutions to the wave equation in the more general
setting of conic manifolds.

Given these regularity results, it is now of interest to try to understand the
related decay (or dispersive) properties of solutions.  Such properties are
often important in studying related nonlinear wave equations as the dispersive
effect of the linear evolution can limit nonlinear interactions.  Mathematically
speaking, this effect gives rise to a family of space-time integrability
inequalities known as \emph{Strichartz estimates}.  These estimates take the
form
\begin{equation}\label{eqn:strich}
  \| u \|_{L^p(\RR;L^q(\cone))}
  \lesssim \left\| f \right\|_{\dot{H}^\gamma(\cone)} + \left\| g \right\|_{
    \dot{H}^{\gamma-1}(\cone)}, 
\end{equation}
where the triple $(p,q,\gamma)$ satisfies the scale invariant condition
\begin{equation}\label{eqn:triplescale}
  \frac{1}{p} + \frac{2}{q} = 1 -\gamma
\end{equation}
and the admissibility requirement
\begin{equation}\label{eqn:tripleadmiss}
  \frac{1}{p} + \frac{1}{2q} \leq \frac{1}{4} .
\end{equation}
The indices are always assumed to satisfy $\gamma \geq 0$ and $2 \leq p,q \leq
\infty$.  Additionally, the triple $\left(4,\infty,\frac{3}{4}\right)$ is
excluded since Strichartz estimates are known to fail in this case.

Strichartz inequalities are well-established for constant coefficient wave
equations posed on $\RR^n$ (see Strichartz~\cite{strich77},
Ginibre-Velo~\cite{ginvelo95}, Lindblad-Sogge \cite{LiSo},
Keel-Tao~\cite{keeltao98}, and references contained therein).  However, only
partial progress has been made in establishing these estimates for solutions on
manifolds, domains, or singular spaces such as cones.  In the last case, the
conic singularity affects the flow of energy and complicates many of the known
techniques for establishing these inequalities.  Nonetheless, in~\cite{ford} the
second author developed a representation of the fundamental solution to the
Schr\"odinger equation on $\cone$ and used it to prove the analogous Strichartz
estimates for that equation.  The present authors together with Sebastian Herr
then used this theorem to obtain estimates for the Schr\"odinger equation on
polygonal domains in~\cite{BFHM}.  The main idea in~\cite{ford} was to use the
functional calculus on cones developed by Cheeger~\cite{C79,C83} to calculate an
explicit representation of the Schwartz kernel of $\exp(it\Delta_\g)$.  In
particular, it was shown that the kernel is uniformly bounded by $|t|^{-1}$, and
hence there is the dispersive estimate
\begin{equation}\label{eqn:schroddisp}
  \left\|\exp\!\left(it\Delta_\g\right) f \right\|_{L^\infty(\cone)} \lesssim
  |t|^{-1}\left\|f\right\|_{L^1(\cone)}.
\end{equation}
Such a dispersive estimate is the key to establishing the full range of
Strichartz inequalities for the Schr\"odinger equation.

Explicit representations of the fundamental solution to the wave equation on
$\cone$ were developed by Cheeger and Taylor in~\cite[Section 4]{CT82no2}.
Unlike the fundamental solution to the Schr\"odinger equation, however, the
fundamental solution to the wave equation is unbounded, and as a consequence one
must rework the $L^1 \To L^\infty$ dispersive estimates.  Even on $\RR^2$, there
is no estimate analogous to \eqref{eqn:schroddisp} which is valid for any choice
of initial data, regardless of whether derivatives are incorporated to ensure
scale invariance.  One way to circumvent this problem is to prove $L^1 \To
L^\infty$ estimates for frequency localized solutions, showing that whenever the
initial data $(f,g)$ is spectrally localized to frequencies near $\mu>0$ there
is a replacement for the analogue of \eqref{eqn:schroddisp}.  Namely, Strichartz
estimates may be proved if one shows the following.

\begin{conjecture}
  Suppose $\beta\!\left(\mu^{-1}\sqrt{\smash[b]{\Delta_\g}}\right) f = f$ and
  $\beta\!\left(\mu^{-1}\sqrt{\smash[b]{\Delta_\g}}\right) g = g$ for some
  smooth cutoff $\beta$ with $\supp(\beta) \subset \left(\frac{1}{4}, 4\right)$.
  Then
  \begin{align}
    \left\| \calU(t) g \right\|_{L^\infty(\cone)} &\lesssim \mu \left(1+\mu
      |t|\right)^{-1/2}
    \left\|g\right\|_{L^1(\cone)} \label{eqn:sinedisp} \\
    \left\| \calUdot(t)f \right\|_{L^\infty(\cone)}&\lesssim \mu \left(1+\mu |t|
    \right)^{-1/2}\left(\mu \left\|f\right\|_{L^1(\cone)} + \left\| \nabla_\g
        f\right\|_{L^1(\cone)} \right), \label{eqn:cosinedisp}
  \end{align}
  where we use the abbreviations
  \begin{equation}\label{eqn:abbreviate}
    \calU(t) \defeq
    \frac{\sin\left(t\sqrt{ \smash[b]{\Delta_\g}
        }\right)}{\sqrt{\smash[b]{\Delta_\g}}} 
    \qquad 
    \text{and} \qquad \calUdot(t) \defeq \cos\left(t\sqrt{\Delta_\g}\right) .
  \end{equation}
\end{conjecture}

Our main theorem in this work is that the former estimate \eqref{eqn:sinedisp}
holds and, by making use of the Hilbert transform, that this estimate is
sufficient to yield the full range of Strichartz estimates.
\begin{theorem}\label{thm:strich}
  Suppose $(f,g) \in \dot{H}^\gamma\!\left(\cone\right)\times
  \dot{H}^{\gamma-1}\!\left(\cone\right)$.  Then,
  \begin{align}
    \left\|\calU(t)g\right\|_{L^p(\RR;L^q(\cone))}&\lesssim \left\|
      g\right\|_{\dot{H}^{\gamma-1}(\cone)} , \label{eqn:sinestrich} \\
    \left\|\calUdot(t)f\right\|_{L^p(\RR;L^q(\cone))}&\lesssim \left\| f
    \right\|_{\dot{H}^\gamma(\cone)}, \label{eqn:cosinestrich}
  \end{align}
  for any triple $(p,q,\gamma)$
  satisfying~\eqref{eqn:triplescale},~\eqref{eqn:tripleadmiss}, and $2\leq p,q
  \leq \infty$ provided $(p,q,\gamma) \neq \left(4,\infty,\frac{3}{4}\right)$.
  Hence, solutions $u$ to the wave equation IVP \eqref{eqn:wave} satisfy the
  Strichartz estimates~\eqref{eqn:strich}.
\end{theorem}
\noindent We define the homogeneous Sobolev spaces appearing in this theorem via
the spectral resolution of the Laplacian; the details are discussed in Section
\ref{sec:Spaces}.

Using duality and an application of the Christ-Kiselev lemma \cite{ChKi}, we
will also establish inhomogeneous Strichartz estimates.
\begin{corollary}\label{thm:inhomogcorr}
  Suppose $2 \leq p,q < \infty$ and $2 \leq \tilde{p},\tilde{q}< \infty$ satisfy
  \eqref{eqn:tripleadmiss} and that
  \begin{equation}\label{eqn:inhomogadmiss}
    \frac{1}{p} + \frac{2}{q} =\frac{1}{\tilde{p}'}+\frac{2}{\tilde{q}'}-2 =
    1-\gamma 
  \end{equation}
  with $(\cdot)'$ denoting the H\"older-dual exponent, i.e.~ $\frac{1}{p} +
  \frac{1}{p'} = 1$.  If the inhomogeneity $F$ is in $L^{\tilde{p}'}\!\left(\RR;
    L^{\tilde{q}'}\!\left(\cone\right)\right)$ and
  \begin{equation*}
    w(t,r,\theta) = \int_{-\infty}^t
    \left(\calU(t-s)F(s,\cdot)\right)\!(r,\theta)\,ds ,
  \end{equation*}
  then
  \begin{multline}\label{eqn:globalinhom}
    \| w \|_{L^p(\RR;L^q(\cone))} + \left\| \left(w,\prtl_t w\right)
    \right\|_{L^\infty(\RR;\dot{H}^\gamma(\cone) \times
      \dot{H}^{\gamma-1}(\cone)) }
    \\
    \mbox{} \lesssim \| F \|_{L^{\tilde{p}'}(\RR; L^{\tilde{q}'}(\cone))} .
  \end{multline}
\end{corollary}
\noindent Note that by localizing the driving force to $[0,\infty)$, we obtain
estimates for $w$ satisfying $\left(D^2_t -\Delta_\g\right)w = F$ with vanishing
initial data.

As a byproduct of the proof of Theorem~\ref{thm:strich}, we also show local
estimates on solutions which instead involve inhomogeneous Sobolev spaces on the
right-hand side. These estimates will play a role in Section~\ref{sec:apps},
where local estimates on planar domains are developed.

\begin{corollary}\label{thm:localcorr}
  Let $u$ be a solution to the inhomogeneous problem,
  \begin{equation}
    \left\{
      \begin{aligned}
        \left(D_t^2 - \Delta_\g \right) u(t,r,\theta) &= F(t,r,\theta) \\
        u(0,r,\theta) &= f(r,\theta) \\
        \prtl_t u (0,r,\theta) &= g(r,\theta) . \\
      \end{aligned}
    \right.
  \end{equation}
  Suppose that $2 \leq p,q < \infty$ and $2 \leq \tilde{p},\tilde{q}< \infty$
  satisfy \eqref{eqn:tripleadmiss} and \eqref{eqn:inhomogadmiss}.  Then for some
  implicit constant depending on $T$,
  \begin{multline}\label{eqn:localstrich}
    \| u \|_{L^p([-T,T];L^q(\cone))} +\| (u,\prtl_t u)
    \|_{L^\infty([-T,T];H^\gamma(\cone) \times H^{\gamma-1}(\cone)) }\\ \lesssim
    \|(f,g)\|_{H^\gamma(\cone)\times H^{\gamma-1}(\cone)} + \| F
    \|_{L^{\tilde{p}'}([-T,T]; L^{\tilde{q}'}(\cone))},
  \end{multline}
  whenever the right hand side is finite.
\end{corollary}

The expressions for the Schwartz kernels of $\calU(t)$ derived by Cheeger and
Taylor vary depending on into which of three regions of the spacetime $\bbR
\times C(\bbS^1_\rho) \times C(\bbS^1_\rho)$ their arguments fall; these regions
are
\begin{equation}
  \begin{aligned}
    \label{eqn:2}
    \text{Region I} &\defeq \left\{(t,r_1,\theta_1,r_2,\theta_2) : 0 < t <
      d_\g\!\left((r_1,\theta_1),(r_2,\theta_2)\right) \right\} , \\
    \text{Region II} &\defeq \left\{ (t,r_1,\theta_1,r_2,\theta_2) :
      d_\g\!\left((r_1,\theta_1),(r_2,\theta_2)\right) < t < r_1 + r_2
    \right\}  , \text{ and} \\
    \text{Region III} &\defeq \Big\{ (t,r_1,\theta_1,r_2,\theta_2) : t > r_1 +
    r_2 \Big\} .
  \end{aligned}
\end{equation}
Here, the Riemannian distance function $d_\g$ is (see~\cite[(3.41)]{CT82})
\begin{equation}
  d_\g\!\left((r_1,\theta_1),(r_2,\theta_2)\right) = \begin{cases}
    \left(r_1^2+r_2^2 -2r_1r_2\cos(\theta_1-\theta_2)\right)^{1/2}, &
    |\theta_1-\theta_2|\leq \pi \\
    r_1+r_2, & |\theta_1-\theta_2|\geq \pi,
  \end{cases}
\end{equation}
provided the angular coordinates $\theta_1$ and $\theta_2$ are chosen so that
$|\theta_1-\theta_2|$ gives the distance between the two points on the circle.
See Figure~\ref{fig:regions} for a heuristic sketch of these regions in $\bbR^2$
at a point in time when Region III is nonempty.  Region I is the part of
spacetime in which the propagator is identically zero owing to finite speed of
the propagation of supports.  Region II is the regime in which waves propagate
as they would on a smooth manifold, i.e.\ the region in which there has been no
interaction between the main front and the cone tip.  Finally, Region III is the
region in which there has been an interaction between the main front and the
cone point.  The singularities, i.e.~ wavefront set, of the propagators in the
transition between Regions I and II are entirely ``geometric", to use the
terminology of Melrose and Wunsch~\cite{melrosewunsch}, which is to say that
they propagate via geodesic flow.  Those in the transition between Regions II
and III can be either geometric or ``diffractive".  Heuristically speaking, the
geometric singularities in this transition are the limits of the geometric
singularities in the transition between Regions I and II, and the diffractive
singularities are those emerging radially from the cone point after a
singularity has entered.

\begin{figure}
  \centering
  \begin{tikzpicture}
    \node at (0,0) {$\bullet$}; \draw (-2,0) circle (3cm); \draw (0,0) circle
    (1cm); \draw (-4,0) node {II}; \draw (-6,0) node {I}; \draw (.5,0) node
    {III}; \draw (-2cm-3pt,-3pt) -- (-2cm+3pt,3pt); \draw (-2cm-3pt,3pt) --
    (-2cm+3pt,-3pt);
  \end{tikzpicture}
  \label{fig:regions}
  \caption{The Regions of behavior; ``$\times$'' denotes the initial pole of the
    source and ``$\bullet$'' the cone point}
\end{figure}
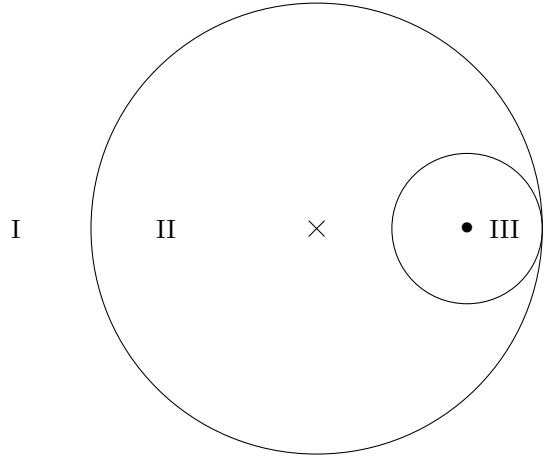

We now discuss the formulae for the Schwartz kernel of $\calU(t)$ in the
different regions of spacetime.  As we remarked, in Region I
\begin{equation}
  \label{eqn:3}
  K_{\calU(t)}^\text{I}(r_1,\theta_1;r_2,\theta_2) \equiv 0 .
\end{equation}
In Region II, it is given by
\begin{multline}
  \label{eqn:4}
  K_{\calU(t)}^\text{II}(r_1,\theta_1;r_2,\theta_2) \\
  \mbox{} = \frac{1}{2\pi} \sum_{j} \left[ t^2 - r_1^2 - r_2^2 + 2r_1 r_2
    \cos\left( \theta_1 - \theta_2 - j \cdot 2\pi\rho \right)
  \right]^{-\frac{1}{2}} ,
\end{multline}
where the index $j$ ranges over integers such that
\begin{equation}
  \label{eqn:6}
  0 < |\theta_1 - \theta_2 - j \cdot 2\pi\rho| < \cos^{-1}\left( \frac{r_1^2 +
      r_2^2 - t^2}{2r_1r_2} \right).
\end{equation}
In Region III, it is given by
\begin{multline}
  \label{eqn:5}
  K_{\calU(t)}^\text{III}(r_1,\theta_1;r_2,\theta_2) \\
  \mbox{} = \frac{1}{2\pi} \sum_{j} \left[ t^2 - r_1^2 - r_2^2 + 2r_1 r_2
    \cos\left( \theta_1 - \theta_2 - j \cdot 2\pi\rho \right)
  \right]^{-\frac{1}{2}} \hspace*{2cm} \\
  \mbox{} - \frac{1}{4\pi^2 \rho} \int_{0}^{\cosh^{-1}\left( \frac{t^2 - r_1^2 -
        r_2^2}{2r_1 r_2} \right)} \left[t^2 -
    r_1^2 - r_2^2 - 2 r_1 r_2 \cosh(s) \right]^{-\frac{1}{2}} \\
  \mbox{} \times \left\{ \frac{ \sin\!\left[ \frac{\pi + \theta_1 -
          \theta_2}{\rho} \right]}{\cosh\!\left[ \frac{s}{\rho} \right] -
      \cos\!\left[ \frac{\theta_1 - \theta_2 + \pi}{\rho} \right]} + \frac{
      \sin\!\left[ \frac{\pi - (\theta_1 - \theta_2)}{\rho}
      \right]}{\cosh\!\left[ \frac{s}{\rho} \right] - \cos\!\left[
        \frac{\theta_1 - \theta_2 - \pi}{\rho} \right]}\right\} ds ,
\end{multline}
where $j$ now ranges over
\begin{equation}
  \label{eqn:7}
  0 < |\theta_1 - \theta_2 - j \cdot 2\pi\rho| < \pi .
\end{equation}

The remainder of the paper is organized as follows.  In Section
\ref{sec:Spaces}, we review Cheeger's functional calculus for metric cones and
define the corresponding homogeneous Sobolev spaces.  Section \ref{sec:ests}
addresses the Strichartz estimates in Theorem \ref{thm:strich}.  In Section
\ref{sec:apps}, we conclude by exploring applications and extensions of these
estimates. Specifically, we discuss Morawetz estimates on the Euclidean cone as
well as Strichartz estimates on wedges, polygons, and Euclidean surfaces with
conic singularities.  We then apply these estimates to the well-posedness of
nonlinear wave equations with initial data of minimal regularity.

\subsection*{Acknowledgements}
It is a pleasure to thank Fabrice Planchon for pointing out the relevance of the
Hilbert transform in establishing Strichartz estimates for the cosine propagator
and the observation that the Morawetz estimate holds immediately from the
existing analysis on the wave equation with the inverse square potential.
Additionally, the foundation of the present work lies in the authors' previous
paper with Sebastian Herr, and we thank him for his insights in this area.

MDB was supported in part by NSF Grants DMS-0801211 and DMS-1001529.  GAF was
partially supported by NSF grant DMS-0636646 and by a Presidential Fellowship at
Northwestern University.  JLM was supported in part by an NSF Postdoctoral
Fellowship in the Department of Applied Physics and Applied Mathematics (APAM)
at Columbia University.  JLM also wishes to acknowledge the hospitality of the
Courant Institute of Mathematical Sciences, where he was a visiting academic
during part of the preparation of this manuscript.

\section{Spectral Theory and Function Spaces}
\label{sec:Spaces}

We begin by briefly recalling Cheeger's functional calculus for metric cones
$C(Y)$.  Using this, we define the homogeneous Sobolev spaces $\dot{H}^s\!\left(
  C(Y) \right)$ appearing in the Strichartz estimates \eqref{eqn:strich}.  We
refer to Cheeger's articles with Taylor \cite{CT82,CT82no2} or the second book
of Taylor's series \cite{Tay2} for an in-depth discussion of the functional
calculus as well as other applications.

\subsection{Cheeger's functional calculus}\label{sec:cheeg-funct-calc}
Let $Y^n$ be a compact, boundaryless Riemannian manifold with metric $\h$, and
let $C(Y) \defeq \bbR_+ \times Y$ be the half-cylinder over $Y$.  To make $C(Y)$
into a metric cone, we equip it with the incomplete Riemannian metric
\begin{equation}
  \label{eq:product-cone-metric}
  \g(r,y) = \upd r^2 + r^2 \, \h(y) .
\end{equation}
The nonnegative Laplacian on $C(Y)$ is thus
\begin{equation}
  \label{eq:11}
  \Delta_\g = - \del_r^2 - \frac{n}{r} \, \del_r + \frac{1}{r^2} \, \Delta_\h ,
\end{equation}
where $\Delta_\h$ is the nonnegative Laplacian on the cross-section $Y$.
Writing $\left\{ \mu_j \right\}_{j=0}^\infty$ for the eigenvalues of $\Delta_\h$
with multiplicity and $\left\{ \varphi_j : Y \To \bbC\right\}_{j=0}^\infty$ for
the corresponding orthonormal basis of eigenfunctions, we define the rescaled
eigenvalues $\nu_j$ by
\begin{equation}
  \label{eq:12}
  \nu_j \defeq \left( \mu_j + \frac{(n-1)^2}{4} \right)^\frac{1}{2} .
\end{equation}
Note that $\mu_0 = 0$ and $\nu_0 = \frac{n-1}{2}$ in our convention.

Henceforth, we take $\Delta_\g$ to be the Friedrichs extension of the above
Laplace-Beltrami operator on functions.  As is well-known, suitable functions $G
: \RR \To \bbC$ give rise to operators $G(\Delta_\g)$ via the spectral theorem.
By taking advantage of the product structure of the metric $\g(r,y)$ and
separation of variables, Cheeger showed that the Schwartz kernel of
$G(\Delta_\g)$, which we will write as $K_{G(\Delta_\g)}$, has the form
\begin{equation}
  \label{eq:13}
  K_{G(\Delta_\g)}(r_1,y_1;r_2,y_2) = (r_1r_2)^{-\frac{n-1}{2}} \sum_{j=0}^\infty
  \tilde{K}_{G(\Delta_\g)}(r_1,r_2,\nu_j) \, \varphi_j(y_1) \,
  \overline{\varphi_j(y_2)} ,
\end{equation}
where the radial coefficient $\tilde{K}_{G(\Delta_\g)}(r_1,r_2,\nu_j)$ is given
by
\begin{equation}\label{eq:14}
  \tilde{K}_{G(\Delta_\g)}(r_1,r_2,\nu) \defeq \int_{0}^\infty
  G(\lambda^2) \, J_\nu(\lambda r_1) \, J_\nu(\lambda r_2) \; \lambda \,
  d\lambda .
\end{equation}
Here, $J_\nu(z)$ is the Bessel function of order $\nu$,
\begin{equation*}
  \label{eq:15}
  J_\nu(z) \defeq \sum_{j=0}^\infty \frac{(-1)^j}{j! \, \Gamma(\nu + j + 1)}
  \left(\frac{z}{2}\right)^{\nu + 2j}.
\end{equation*}

We can view the formulae \eqref{eq:13} and \eqref{eq:14} as consequences of a
``factoring'' of the spectral measure on $C(Y)$ into tangential components and
radial components.  Indeed, the product $J_{\nu_j}(\l r)\varphi_j(y)$ is a
solution to the Helmholtz equation
\begin{equation*}
  \left(\Delta_\g - \l^2\right) \left(J_{\nu_j}(\l r)\varphi_j(y) \right) = 0.
\end{equation*}
This naturally leads one to consider the Hankel transform of order $\nu_j$,
\begin{equation}
  \label{eq:5}
  \mathcal{H}_{\nu_j}\!\left[b(r)\right]\!(\lambda) \defeq \int_{0}^\infty b(r)
  \, J_{\nu_j}(\lambda r) \; r \, dr ,
\end{equation}
a unitary map $L^2(\RR_+, r\,dr) \To L^2(\RR_+, \l\,d\l)$ satisfying
$\mathcal{H}_{\nu_j}\circ \mathcal{H}_{\nu_j}=\Id$ (see~\cite[Ch.~ 8,
Proposition 8.1]{Tay2}).  This gives rise to a unitary isomorphism
\begin{equation}
  \label{eq:1}
  \bm{\mathcal{H}} : L^2\!\left(C(Y)\right) \stackrel{\cong}{\To}
  L^2\!\left(\bbR_+,\lambda \, d\lambda ; \ell^2\!\left(\bbZ_{\geq 0}\right)
  \right)
\end{equation}
which acts on functions $a(r,y)$ in $L^2(C(Y))$ by
\begin{equation}
  \label{eq:3}
  \bm{\mathcal{H}}\!\left[a\right]\!(\lambda) \defeq \left(
    \left(\mathcal{H}_{\nu_j} \circ \Pi_{j} \right)\! \left[ a(\cdot,\cdot)
    \right]\!(\lambda) \right)_{j \in \bbZ},
\end{equation}
where $\Pi_{j}$ is the projection onto the $j$-th eigenfunction of the
tangential Laplacian $\Delta_\h$,
\begin{equation}
  \label{eq:4}
  \Pi_{j}\!\left[a(\cdot,\cdot)\right]\!(r) \defeq \int_Y a(r,y) \,
  \overline{\varphi_j(y)} \; d\h .
\end{equation}
For further details, we refer the reader to Cheeger and Taylor \cite{CT82}.

\subsection{The Sobolev spaces}
\label{sec:sobol-spac}

We now define the function spaces needed for our analysis.
\begin{definition}
  For $s \geq 0$, we define the \emph{homogeneous Sobolev spaces}
  $\dot{H}^s\!\left( C(Y) \right)$ to be the completion of
  $\mathcal{C}^\infty_c(C(Y))$ in the topology induced by the inner product
  \begin{equation}
    \label{eq:18}
    \left< u,v \right>_{\dot{H}^s} \defeq \left< \Delta_\g^\frac{s}{2} u,
      \Delta_\g^\frac{s}{2} v \right>_{L^2} .
  \end{equation}
  We define the homogeneous Sobolev spaces of negative order by duality, i.e.~
  for $s > 0$,
  \begin{equation}
    \label{eq:19}
    \dot{H}^{-s}\!\left( C(Y) \right) \defeq \left( \dot{H}^s\!\left(C(Y)\right)
    \right)',
  \end{equation}
  equipped with the dual norm.
\end{definition}

When $s>0$, functions in $\Dom\!\left(\Delta_\g^{-\frac{s}{2}}\right)$ define
elements in $\dot{H}^{-s}(C(Y))$ via the usual $L^2$-pairing.  Indeed, if $\psi
\in \Dom\!\left(\Delta_\g^{-\frac{s}{2}}\right)$ and $f \in \dot{H}^{s}(C(Y))
\cap \Dom\!\left(\Delta_\g^{\frac{s}{2}}\right)$, then the functional calculus
shows that $\Delta_\g^{-\frac{s}{2}}\psi \in
\Dom\!\left(\Delta_\g^{\frac{s}{2}}\right)$; hence
\begin{equation*}
  \left| \left\langle f, \psi \right\rangle_{L^2} \right| = \left| \left\langle
      \Delta_\g^{\frac{s}{2}} f, \Delta_\g^{-\frac{s}{2}}\psi
    \right\rangle_{L^2} \right|
  \leq \left\|f\right\|_{\dot{H}^{s}} 
  \left\|\Delta_\g^{-\frac{s}{2}}\psi\right\|_{L^2}.
\end{equation*}
This also shows that if $\Psi \in \dot{H}^{-s}(C(Y))$ is defined by $\Psi(f) =
\left\langle f, \psi \right\rangle$, then we can take $f =
\Delta_\g^{-\frac{s}{2}} \psi$ to see that $\left\|\Psi\right\|_{\dot{H}^{-s}}=
\left\| \Delta_\g^{-\frac{s}{2}} \psi \right\|_{L^2}$.  Moreover,
$\Dom\!\left(\Delta_\g^{-\frac{s}{2}}\right)$ is dense in $\dot{H}^{-s}(C(Y))$.
Indeed, given any $\Psi \in \dot{H}^{-s}(C(Y))$, the Riesz representation
theorem shows that there exists $\phi \in \dot{H}^{s}(C(Y))$ such that
$\Psi(f)=\langle f, \phi \rangle_{\dot{H}^{s}}$ and $\|\Psi\|_{\dot{H}^{-s}}=
\|\phi\|_{\dot{H}^s}$.  Since $\mathcal{C}^\infty_c(C(Y))$ is dense in
$\Dom\!\left(\Delta_\g^{\frac{s}{2}}\right)$, there exists a sequence
$\{\phi_n\}_{n=1}^\infty \subseteq \mathcal{C}^\infty_c(C(Y))$ such that
$\|\phi_n - \phi\|_{\dot{H}^s} \To 0$.  Thus $\psi_n \defeq \Delta_\g^{s}
\phi_n$ is well defined and provides a sequence of functions in
$\Dom\!\left(\Delta_\g^{-\frac{s}{2}}\right)$ which approximates $\Psi$ in
$\dot{H}^{-s}(C(Y))$.

As a result of this density property, we may take $Y$ to be the 1-dimensional
manifold $\mathbb{S}^1_\rho$ to see that it suffices to show
Theorem~\ref{thm:strich} under the stronger assumption
\begin{equation}\label{eqn:initialdata}
  (f,g) \in \left[ \dot{H}^\gamma\!\left(\cone\right) \cap
    \Dom\!\left(\Delta_\g^{\frac{\gamma}{2}}\right) \right]
  \times
  \left[ \dot{H}^{\gamma-1}\!\left(\cone\right) \cap
    \Dom\!\left(\Delta_\g^{\frac{\gamma-1}2}\right) \right] . 
\end{equation}

We define the inhomogeneous Sobolev spaces similarly.
\begin{definition}
  For all real $s$, we define the \emph{inhomogeneous Sobolev spaces}, denoted
  $H^s\!\left(C(Y)\right)$, to be the closure of $ \mathcal{C}^\infty_c\!\left(
    C(Y) \right)$ in the topology induced by the inner product
  \begin{equation}
    \label{eq:top}
    \left< u,v \right>_{H^s} \defeq \left< \left(\Id +
        \Delta_\g\right)^\frac{s}{2} u , \left(\Id +
        \Delta_\g\right)^\frac{s}{2} v \right>_{L^2} .
  \end{equation}
\end{definition}

\section{Dispersive Estimates}
\label{sec:ests}
In this section, we prove Theorem~\ref{thm:strich} using an explicit formula for
the kernel of the sine propagator $\calU(t)$.  As discussed in the introduction,
we will take a Littlewood-Paley decomposition of the solution to establish
frequency-localized dispersive estimates.  These will allow for a regularization
of $K_{\calU(t)}$ to unit frequency that will overcome the complications coming
from unboundedness of the fundamental solution.

\subsection{Reduction to dispersive estimates}
\label{sec:reduction}
Let $\left\{ \beta_k \right\}_{k \in \mathbb{Z}}$ be a collection of smooth
cutoffs satisfying
\begin{equation}\label{eqn:L-Pcutoffs}
  \sum_{k=-\infty}^\infty \beta_k(\zeta) \equiv 1, \quad \beta_k(\zeta) \defeq
  \beta_0\!\left(2^{-k}\zeta\right) \! , \quad \text{and} \quad
  \supp\!\left(\beta_0\right) 
  \subset \left( \frac{1}{\sqrt{2}},2\sqrt{2} \right) .
\end{equation}
We define the associated Littlewood-Paley frequency cutoffs
$\beta_k\!\left(\sqrt{\smash[b]{\Delta_\g}}\right)$ using the functional
calculus reviewed in Section \ref{sec:Spaces}, and we see immediately from the
definition that
\begin{equation*}
  \sum_{k=-\infty}^\infty \beta_k\!\left(\sqrt{\smash[b]{\Delta_\g}}\right) =
  \Id : L^2\!\left(\cone\right) \To L^2\!\left(\cone\right).
\end{equation*}
More generally, these Littlewood-Paley cutoffs satisfy a range of squarefunction
estimates, which we summarize in the following proposition.
\begin{proposition}
  \label{lem:sqfnests}
  Let $1 < q < \infty$.  For elements $a \in L^q\!\left(\cone\right)$, we have
  \begin{equation}\label{eqn:L-P estimate}
    \left\| \left(\sum_{k =-\infty}^\infty
        \left|\beta_k\!\left(\sqrt{\smash[b]{\Delta_\g}}
          \right) a \right|^2\right)^{\frac{1}{2}} \right\|_{L^q(\cone)} \approx
    \left\|a\right\|_{L^q(\cone)}  ,
  \end{equation}
  with implicit constants depending only on $q$.
\end{proposition}

The proof of this proposition is implicit in the arguments in Section 4 of
\cite{BFHM}, the authors' previous paper with Herr, where the squarefunction
estimates are shown for functions on a Euclidean surface with conical
singularities.  Namely, \eqref{eqn:L-P estimate} follows from the spectral
multiplier theorem of Alexopolous~\cite{alexopolous}, which is valid in any
context where the heat kernel satisfies Gaussian upper bounds. A result of
Grigor'yan~\cite{grigoryan} shows that such bounds are true provided they are
satisfied along the diagonal. On the cone $\cone$, these on-diagonal bounds
follow from an explicit formula for the heat kernel.  See Section 4 of
\cite{BFHM} for a more thorough discussion.

Returning to the wave equation, suppose $u$ is a solution to the IVP
\eqref{eqn:wave} with initial data satisfying~\eqref{eqn:initialdata}.  We
define its frequency decomposition to be
\begin{equation}
  \label{eq:freq-decomp}
  u_k(t,\cdot) \defeq
  \beta_k\!\left(\sqrt{\smash[b]{\Delta_\g}}\right)u(t,\cdot) \quad 
  \text{and} \quad \left(f_k,g_k\right) \defeq
  \left(\beta_k\!\left(\sqrt{\smash[b]{\Delta_\g}}\right)f,
    \beta_k\!\left(\sqrt{\smash[b]{\Delta_\g}}\right)g\right) .
\end{equation}
Since the frequency cutoffs commute with the Laplacian, the frequency localized
solutions $\left\{u_k\right\}_{k\in\bbZ}$ satisfy the collection of IVPs
\begin{equation}
  \label{eq:freq-local-IVPs}
  \left\{
    \begin{aligned}
      \left(D_t^2 - \Delta_\g \right) u_k(t,r,\theta) &= 0 \\
      u_k(0,r,\theta) &= f_k(r,\theta) \\
      \prtl_t u_k(0,r,\theta) &= g_k(r,\theta) . \\
    \end{aligned}
  \right.
\end{equation}
As a consequence of the squarefunction estimates \eqref{eqn:L-P estimate} and
Minkowski's inequality we have
\begin{equation*}
  \|u\|_{L^p(\RR;L^q(\cone))} \lesssim \left(\sum_{k=-\infty}^\infty
    \left\|u_k\right\|_{L^p(\RR;L^q(\cone))}^2 \right)^{\frac{1}{2}}.
\end{equation*}
Furthermore, we note that the operator norm of
\begin{equation*}
  2^{-ks} \, \Delta_\g^{\frac{s}{2}} \, \beta_k\!\left(\sqrt{\smash[b]{\Delta_\g}}\right) 
\end{equation*}
on $L^2\!\left(\cone\right)$ is uniformly bounded in $k$, which implies that
\begin{equation}\label{eqn:l2hs}
  2^{k\gamma} \left\|f_k\right\|_{L^2} + 2^{k(\gamma-1)}
  \left\|g_k\right\|_{L^2} \lesssim \left\|\Delta_\g^{\frac{\gamma}{2}}f_k
  \right\|_{L^2} + \left\|\Delta_\g^{\frac{\gamma-1}{2}}g_k \right\|_{L^2}.
\end{equation}
Therefore, if we show the collection of frequency-localized Strichartz estimates
\begin{equation}\label{eqn:freqloc}
  \left\|u_k\right\|_{L^p(\RR;L^q(\cone))} \lesssim 2^{k\gamma}
  \left\|f_k\right\|_{L^2(\cone)} + 2^{k(\gamma-1)}
  \left\|g_k\right\|_{L^2(\cone)},
\end{equation}
then the desired Strichartz estimates~\eqref{eqn:strich} will follow
from~\eqref{eqn:l2hs} and the bound
\begin{multline}\label{eqn:orthog}
  \sum_{k=-\infty}^\infty \left( \left\|\Delta_\g^{\frac{\gamma}{2}}f_k
    \right\|^2_{L^2} + \left\|\Delta_\g^{\frac{\gamma-1}{2}}g_k \right\|^2_{L^2}
  \right) \\
  = \sum_{k=-\infty}^\infty \left(
    \left\|\beta_k\!\left(\sqrt{\smash[b]{\Delta_\g}}\right)
      \Delta_\g^{\frac{\gamma}{2}}f \right\|_{L^2}^2 + \left\|
      \beta_k\!\left(\sqrt{\smash[b]{\Delta_\g}}\right)
      \Delta_\g^{\frac{\gamma-1}{2}}g \right\|_{L^2}^2\right) \\
  \lesssim \left\|f\right\|_{\dot{H}^{\gamma}}^2 +
  \left\|g\right\|_{\dot{H}^{\gamma-1}}^2 .
\end{multline}

\begin{remark}
  We also note at this stage that the homogeneous estimates in
  Corollary~\ref{thm:localcorr} follow by a slight modification of these
  arguments.  A corresponding adjustment in the proof of Corollary
  \ref{thm:inhomogcorr} will handle the inhomogeneous inequality.  For local
  estimates, we instead take $\tilde{u}_0 \defeq \sum_{k\leq 0} u_k$, and we
  observe by Sobolev embedding and the fact that $\tilde{u}_0$ is supported at
  low frequency that
  \begin{equation*}
    \begin{aligned}
      \left\|\tilde{u}_0\right\|_{L^p([-T,T];L^q(\cone))} &\lesssim \sup_{-T\leq
        t\leq T} \left\|\tilde{u}_0(t)\right\|_{H^{\gamma+\frac 12}(\cone)}\\
      &\lesssim \left\|f\right\|_{H^{\gamma}(\cone)} +
      \left\|g\right\|_{H^{\gamma-1}(\cone)}.
    \end{aligned}
  \end{equation*}
  Now, note that we can replace $\Delta_\g$ by $\Id+\Delta_\g$
  in~\eqref{eqn:l2hs} when $k \geq 1$.  Thus if~\eqref{eqn:freqloc} holds, we
  can apply reasoning similar to the above to prove Corollary
  \ref{thm:localcorr}.
\end{remark}

We now reduce the collection of frequency-localized Strichartz estimates
\eqref{eqn:freqloc} to a single Strichartz estimate for initial data with
frequency localized to unit scale.
\begin{lemma}
  Suppose $f, g \in L^2\!\left(\cone\right)$ are functions satisfying
  \begin{equation*}
    \beta\!\left(\sqrt{\smash[b]{\Delta_\g}}\right) f = f \quad \text{and} \quad
    \beta\!\left(\sqrt{\smash[b]{\Delta_\g}}\right) g = g
  \end{equation*}
  for some smooth cutoff $\beta \in \mathcal{C}^\infty_c(\bbR)$ supported in
  $\left(\frac{1}{4},4\right)$.  Suppose also that the corresponding solution to
  the wave equation IVP \eqref{eqn:wave} satisfies the Strichartz estimate
  \begin{equation}\label{eqn:freqlocunit}
    \|u\|_{L^p(\RR;L^q(\cone))} \lesssim \left\|f\right\|_{L^2(\cone)} +
    \left\|g\right\|_{L^2(\cone)}.
  \end{equation}
  Then the estimates~\eqref{eqn:freqloc} hold for all integers $k$.
\end{lemma}
\begin{proof}
  Recall that the wave equation is invariant under the scaling
  \begin{equation*}
    (t,r,\theta) \mapsto \left(\mu^{-1}t, \mu^{-1}r, \theta\right).
  \end{equation*}
  It thus suffices to show that given $h$ in $L^2\!\left(\cone\right)$, the
  rescaled function
  \begin{equation*}
    \left(\beta_k\!\left(\sqrt{\smash[b]{\Delta_\g}}\right)
      h\right)\!\left(2^{-k}r, \theta \right)
  \end{equation*}
  is localized to unit frequency.  Let $\varphi_j(\theta) \defeq
  \left(2\pi\rho\right)^{-\frac{1}{2}} e^{ij\theta/\rho}$ be a normalized
  eigenfunction on $\bbS^1_\rho$.  Without loss of generality, we can assume
  that $h$ takes the form $h(r,\theta)=h_j(r) \, \varphi_j(\theta)$, since any
  $h \in L^2\!\left(\cone\right)$ can be realized as a sum of such functions.
  Therefore,
  \begin{multline*}
    \left(\beta_k\!\left(\sqrt{\smash[b]{\Delta_\g}}\right)h\right)
    (2^{-k}r,\theta) \\
    = \varphi_j (\theta)\int_{0}^\infty \int_{0}^\infty \beta_0 (2^{-k} \lambda)
    \, J_{\nu_j}(2^{-k} \lambda r) \, J_{\nu_j}(\lambda s) \, h_j(s) \; s \, ds
    \, \lambda \, d\lambda .
  \end{multline*}
  We now make the change of variables $(\lambda,s) \mapsto
  \left(2^{k}\lambda,2^{-k}s \right)$ in the integral on the right, producing
  \begin{equation*}
    \varphi_j(\theta) \int_{0}^\infty \int_{0}^\infty
    \beta_0(\lambda) \, J_{\nu_j}(\lambda r) \, J_{\nu_j}(\lambda s) \,
    h_j (2^{-k}s) \; s \, ds \, \lambda \, d\lambda .
  \end{equation*}
  This is $\left( \beta_0\!\left(\sqrt{\smash[b]{\Delta_\g}}\right)
    h(2^{-k}\cdot,\cdot)\right)\!(r,\theta)$, which is manifestly of the desired
  form.
\end{proof}

Hence, tracing back through these reductions, we see that the main theorem is a
consequence of the following two estimates for initial data $(f,g)$ that is
localized to unit frequency and any smooth cutoff $\beta$ supported in
$\left(\frac{1}{4}, 4\right)$:
\begin{align}
  \left\|\beta\!\left(\sqrt{\smash[b]{\Delta_\g}}\right) \calU(t) \,
    g\right\|_{L^p(\RR;L^q(\cone))}&\lesssim
  \left\|g\right\|_{L^2(\cone)} \label{eqn:sineest} \\
  \left\|\beta\!\left(\sqrt{\smash[b]{\Delta_\g}}\right) \calUdot(t) \,
    f\right\|_{L^p(\RR;L^q(\cone))}&\lesssim \left\|\sqrt{\smash[b]{\Delta_\g}}
    \, f\right\|_{L^2(\cone)} .\label{eqn:cosineest}
\end{align}
\begin{remark}
  \label{rem:Hilbert-transform}
  The main theorem follows from \eqref{eqn:sineest} alone, for the Strichartz
  estimate \eqref{eqn:cosineest} for the cosine evolution operator is a
  consequence of that for the sine evolution operator.  To see this, let
  $\mathcal{T}$ denote the operator acting on functions $v:\RR \times \cone \To
  \RR$ by taking the Hilbert transform in the first variable, i.e.~
  \begin{equation}
    (\mathcal{T}v)(t,r,\theta) \defeq \frac 1\pi \cdot
    \text{PV}\int_{-\infty}^\infty \frac{v(s,r,\theta)}{t-s} \; ds
  \end{equation}
  where PV denotes that this is a principal value integral.  Writing
  \begin{equation*}
    \begin{aligned}
      w(t,r,\theta) &\defeq
      \left(\beta\!\left(\sqrt{\smash[b]{\Delta_\g}}\right)
        \calUdot(t) \, f\right)\!(r,\theta) \\
      v(t,r,\theta) &\defeq
      \left(\beta\!\left(\sqrt{\smash[b]{\Delta_\g}}\right) \calU(t) \,
        \sqrt{\smash[b]{\Delta_\g}} \, f \right)\!(r,\theta)
    \end{aligned}
  \end{equation*}
  and $\bm{\mathcal{H}}$ for the spectral resolution of $\Delta_\g$ as in
  Section \ref{sec:Spaces}, we have
  \begin{equation*}
    \begin{aligned}
      \left(\bm{\mathcal{H}}w\right)\!(t,\l) &= \beta(\l)\cos(t\l) \cdot
      \left(\bm{\mathcal{H}}f\right)\!(\l) \\
      &= -\frac{1}{\pi} \cdot \text{PV}\int_{-\infty}^\infty
      \frac{\sin(t\l) \beta(\l) (\bm{\mathcal{H}}f)(\l)}{t-s} \; ds \\
      &= - \mathcal{T} \left[(\bm{\mathcal{H}}v)(t,\l) \right] .
    \end{aligned}
  \end{equation*}
  Thus, $w = - \mathcal{T} v$ since the Hilbert transform commutes with
  $\bm{\mathcal{H}}$.  Our claim then follows from noting that $\mathcal{T}$ is
  bounded as an operator on $L^p\!\left(\RR;L^q\!\left(\cone\right)\right)$,
  implying the estimate
  \begin{equation*}
    \left\|\beta\!\left(\sqrt{\smash[b]{\Delta_\g}}\right) \calUdot(t) \, f
    \right\|_{L^p(\RR;L^q(\cone))} \lesssim 
    \left\|\beta\!\left(\sqrt{\smash[b]{\Delta_\g}}\right) \calU(t) \,
      \sqrt{\smash[b]{\Delta_\g}} \, f \right\|_{L^p(\RR;L^q(\cone))} .
  \end{equation*}
  We note that the bound on $\mathcal{T}$ is a consequence of the
  Calder\'on-Zygmund theory of vector-valued singular integrals; see, for
  instance, Theorem $4.6.1$ in Grafakos' book \cite{grafakos}.  With this remark
  in mind, we now prove that Corollary \ref{thm:inhomogcorr} is a consequence of
  \eqref{eqn:sineest}.
\end{remark}

\begin{proof}[Proof of Corollary \ref{thm:inhomogcorr}]
  Let
  \begin{equation*}
    \mathcal{S}_1 \defeq \frac{\beta\!\left(\sqrt{\smash[b]{\Delta_\g}}\right)
      e^{it\sqrt{\smash[b]{\Delta_\g}}}}{\sqrt{\smash[b]{\Delta_\g}}} \quad
    \text{and} \quad \mathcal{S}_2 \defeq
    \beta\!\left(\sqrt{\smash[b]{\Delta_\g}}\right)
    e^{it\sqrt{\smash[b]{\Delta_\g}}} . 
  \end{equation*}
  By Euler's formula and the remark above (changing the form of $\beta$ if
  necessary), we have that if $\beta\!\left(\sqrt{\smash[b]{\Delta_\g}}\right)f
  = f$, then
  \begin{equation*}
    \left\|\mathcal{S}_i f \right\|_{L^p(\RR;L^q(\cone))} \lesssim
    \|f\|_{L^2(\cone)}
  \end{equation*}
  for $i = 1,2$ and any admissible $(p,q)$ in \eqref{eqn:tripleadmiss},
  including $(p,q)=(\infty,2)$.  Now take $F(s,r,\theta)$ to be a smooth
  function, compactly supported in time, such that
  $\beta\!\left(\sqrt{\smash[b]{\Delta_\g}}\right) F(s,\cdot) = F(s,\cdot)$.  By
  duality, if $(\tilde{p},\tilde{q})$ is also admissible then we have
  \begin{equation*}
    \left\|\mathcal{S}_2 \mathcal{S}^*_1 F \right\|_{L^p(\RR;L^q(\cone))} \lesssim
    \|F\|_{L^{\tilde{p}'}(\RR;L^{\tilde{q}'}(\cone))}.
  \end{equation*}
  The Christ-Kiselev lemma \cite{ChKi} and time reversal now give
  \begin{equation*}
    \left\|\int_{-\infty}^t \calU(t-s)F(s)\,ds\right\|_{L^p(\RR;L^q(\cone))}
    \lesssim \|F\|_{L^{\tilde{p}'}(\RR;L^{\tilde{q}'}(\cone))},
  \end{equation*}
  once again including the case $(p,q) = (\infty,2)$.  A slight adjustment of
  these arguments implies that
  \begin{equation*}
    \left\|\int_{-\infty}^t \sqrt{\smash[b]{\Delta_\g}} \,
      \calU(t-s) \, F(s) \, ds \right\|_{L^p(\RR;L^q(\cone))} \lesssim
    \|F\|_{L^{\tilde{p}'}(\RR;L^{\tilde{q}'}(\cone))}.
  \end{equation*}
  Corollary \ref{thm:inhomogcorr} now follows by scaling considerations and
  Littlewood-Paley theory as before.
\end{proof}

It is now well-known via the standard $T T^*$ argument (see for instance Keel
and Tao \cite{keeltao98}, though the techniques originated much earlier) that if
one has a dispersive estimate of the form
\begin{equation}\label{eqn:dispersivesine}
  \left\| \beta\!\left(\sqrt{\smash[b]{\Delta_\g}}\right) \calU(t) \, g
  \right\|_{L^\infty(\cone)} \lesssim ( 1+ t)^{-\frac{1}{2}}
  \left\|g\right\|_{L^1(C(\mathbb{S}^1_\rho))}
\end{equation}
for data $g$ localized to unit frequency, then~\eqref{eqn:sineest} will follow.
The estimate~\eqref{eqn:dispersivesine} in turn follows by establishing bounds
on the supremum of the Schwartz kernel of
$\beta\!\left(\sqrt{\smash[b]{\Delta_\g}}\right) \calU(t)$.  On $\RR^2$, this is
typically accomplished by oscillatory integral methods.  On flat cones,
analogous oscillatory integrals appear to be very difficult to obtain.  Instead,
we work entirely on the spatial domain, treating the Littlewood-Paley cutoffs as
operators which regularize the corresponding kernels.  The dispersive estimates
will then follow by showing that the regularized kernel is essentially bounded
by its average over a set of unit size.

\subsection{Bounds on the ``geometric'' term}
\label{sec:geombounds}

We may unify the formulae in~\eqref{eqn:4} and~\eqref{eqn:5} to obtain a kernel
which is supported in the union of Regions II and III.  It will be written as
the sum of two terms which we (somewhat) informally describe as a ``geometric''
term and a ``diffractive'' term,
\begin{equation}\label{eqn:gplusd}
  K_{\calU(t)}(t,r_1,\theta_1;r_2,\theta_2) =
  K_{\calU(t)}^\geom(t,r_1,\theta_1;r_2,\theta_2) +
  K_{\calU(t)}^\diff(t,r_1,\theta_1;r_2,\theta_2).
\end{equation}
The expression in~\eqref{eqn:4} and the first term in~\eqref{eqn:5} can be
unified to form the geometric term
\begin{eqnarray}
  \label{eqn:Kgeom}
  K_{\calU(t)}^\geom(t,r_1,\theta_1;r_2,\theta_2) =
  \Psi(t,r_1,r_2,\theta_1-\theta_2),
\end{eqnarray}
where $\Psi$ is defined as
\begin{equation}\label{eqn:psidefn}
  \Psi(t,r_1,r_2,\theta) \defeq \sum_{-\pi \leq \theta +j\cdot 2 \pi \rho \leq
    \pi} \left[ t^2 -r_1^2 - r_2^2 + 2r_1r_2 \cos (\theta+j\cdot 2 \pi \rho)
  \right]^{-\frac{1}{2}}_+ .
\end{equation}

\begin{remark}
  The summation in \eqref{eqn:psidefn} is dependent upon the relative location
  of the points in the integral kernel on the cone.  However the total number of
  terms in the sum is no more than $1 + 1/\rho$. Also note that it vanishes for
  $|\theta| > \pi$.
\end{remark}

The diffractive term will be the remaining term in~\eqref{eqn:5}; it is
supported solely in Region III.  To simplify our expression for it, we use the
abbreviations
\begin{alignat*}{3}
  \alpha &\defeq \frac{t^2-r_1^2-r_2^2}{2r_1 r_2} = \frac{t^2-(r_1+r_2)^2}{2r_1
    r_2} + 1 &\hspace*{1.5cm}
  \beta &\defeq \cosh^{-1}(\alpha) \\
  \varphi_1 &\defeq \frac{\pi + \left(\theta_1 -\theta_2\right)}{\rho} &
  \varphi_2 &\defeq \frac{\pi - \left(\theta_1 -\theta_2\right)}{\rho}.
\end{alignat*}
We now write $K_{\calU(t)}^\diff(t, r_1, \theta_1;r_2, \theta_2)$ as
\begin{multline}\label{eqn:diffractivesine}
  K_{\calU(t)}^\diff(t, r_1, \theta_1;r_2, \theta_2) =
  -\frac{\mathbf{1}_{(0,t)}(r_1+r_2)}{4\pi^2 \rho \left(2r_1r_2\right)^{\frac
      12}} \\
  \mbox{} \times \int_0^\beta \left[ \alpha - \cosh(s) \right]^{-\frac 12}
  \left[ \frac{\sin(\varphi_1)}{\cosh\!\left(s/\rho\right) - \cos(\varphi_1)} +
    \frac{\sin(\varphi_2)}{\cosh\!\left(s/\rho \right) - \cos(\varphi_2)}
  \right] ds.
\end{multline}

Let $K_{\beta(\sqrt{\smash[b]{\Delta_\g}})}(r_3,\theta_3;r_2,\theta_2)$ denote
the kernel of $\beta\!\left(\sqrt{\smash[b]{\Delta_\g}}\right)$.  To show the
dispersive estimate~\eqref{eqn:dispersivesine}, we will establish the stronger
bounds
\begin{equation}\label{eqn:gcompose}
  \int_{C(\sph^1_\rho)} \left|
    K_{\beta(\sqrt{\smash[b]{\Delta_\g}})}(r_3,\theta_3;r_2,\theta_2) 
    K_{\calU(t)}^\geom(t,r_1,\theta_1;r_2,\theta_2)\right|
  r_2\,dr_2\,d\theta_2 \lesssim \min\!\left( t,
    t^{-\frac{1}{2}} \right),
\end{equation}
and
\begin{equation}\label{eqn:dcompose}
  \int_{C(\sph^1_\rho)} \left|
    K_{\beta(\sqrt{\smash[b]{\Delta_\g}})}(r_3,\theta_3;r_2,\theta_2)
    K_{\calU(t)}^\diff(t,r_1,\theta_1;r_2,\theta_2)\right| r_2\,dr_2\,d\theta_2
  \lesssim \min\!\left( t, 
    t^{-\frac{1}{2}} \right) .
\end{equation}

Since the heat kernel $e^{-t\Delta_\g}$ satisfies Gaussian upper bounds, as
discussed in Section \ref{sec:reduction}, we may use a theorem of
Davies~\cite[Theorem 3.4.10]{Davies} that provides bounds on the kernel of
Schwartz class functions of the Laplacian.  Namely, we obtain the following
bound on $K_{\beta(\sqrt{\smash[b]{\Delta_\g}})}(r_3,\theta_3;r_2,\theta_2)$,
where we take $N$ to be sufficiently large:
\begin{equation*}
  \left| K_{\beta(\sqrt{\smash[b]{\Delta_\g}})}(r_3,\theta_3;r_2,\theta_2)
  \right| \lesssim 
  \left(1 + d_\g^2(r_3,\theta_3;r_2,\theta_2)\right)^{-N}.
\end{equation*}

On $\bbS^1_\rho \defeq \mathbb{R}/2\pi\rho\mathbb{Z}$, we take coordinates
$\theta\in(-\pi\rho,\pi\rho]$ and assume without loss of generality that
$\theta_1 = 0$.  Let $\varepsilon>0$ be a small parameter such that $\min(\pi,
\pi\rho) \big/ \varepsilon$ is an integer. We tile the portion of the cone
$(0,\infty)_r \times (-\min(\pi, \pi\rho), \min(\pi, \pi\rho)]_\theta$ into
``polar rectangles''
\begin{equation}
  R_{k,\ell} \defeq \left(k\varepsilon,(k+1)\varepsilon\right]_r \times
  \left(\frac{\ell}{k} \varepsilon, \frac {\ell+1}k\varepsilon\right]_\theta,
\end{equation}
where $k> 0$ and
\begin{equation*}
  -\frac{\min(\pi, \pi\rho)\cdot k}{\varepsilon} \leq \ell \leq
  \frac{\min(\pi, \pi\rho) \cdot k}{\varepsilon} -1.
\end{equation*}
When $k=0$, we take
\begin{equation}
  R_{0,1} \defeq (0,\veps]_r \times \left(-\min(\pi, \pi\rho), 0\right]_\theta
  \quad \text{and} \quad R_{0,2} \defeq (0,\veps]_r \times (0, \min(\pi,
  \pi\rho)]_\theta . 
\end{equation}

Observe that the diameter of each of these rectangles is approximately
$\varepsilon$.  Indeed, if $(r,\theta)$ lies in one of these, then
\begin{equation}\label{eqn:diameter}
  \begin{aligned}
    d_\g^2\!\left(r,\theta;k\varepsilon,\frac{\ell \varepsilon}{k} \right) &=
    r^2+k^2 \varepsilon^2
    -2rk\varepsilon\cos\!\left(\theta - \frac{\ell\varepsilon}{k}\right) \\
    &=(r-k\varepsilon)^2 + \mathrm{O}\!\left(rk\varepsilon \cdot
      \frac{\varepsilon^2}{k^2}\right) \\
    &\lesssim \varepsilon^2
  \end{aligned}
\end{equation}
since $\cos(\phi) = 1 + \mathrm{O}(\phi^2)$.  Therefore, if $(r_2,\theta_2) \in
R_{k,\ell}$ we have
\begin{equation*}
  \left(1+ d_\g^2(r_3,\theta_3;r_2,\theta_2)\right)^{-N} \lesssim
  \left(1+d_\g^2\!\left(r_3,\theta_3;k\varepsilon,
      \frac{\ell\varepsilon}{k}\right)\right)^{-N}.
\end{equation*}
We may now bound the left hand side of~\eqref{eqn:gcompose} by a constant
multiple of
\begin{equation}\label{eqn:rectsum}
  \sum_{k,\ell}\left(1+d_\g^2\!\left(r_3,\theta_3;k\varepsilon,
      \frac{\ell\varepsilon}{k}\right)\right)^{-N}  \int_{R_{k,\ell}} \Psi(r_2,
  \theta_2; 
  r_1,\theta_1) \; r_2\,dr_2\, d\theta_2.
\end{equation}
Thus, if we can show that
\begin{equation}\label{eqn:rectbound}
  \int_{R_{k,\ell}} \Psi(t,r_1,\theta_1;r_2,\theta_2) \; r_2\,dr_2\, d\theta_2
  \lesssim \min\!\left(t,t^{-\frac{1}{2}} \right) ,
\end{equation}
then~\eqref{eqn:gcompose} will follow because
\begin{equation}\label{eqn:phibound}
  \sum_{k,\ell} \left(1+d^2_\g\!\left(r_3,\theta_3;k\varepsilon,
      \frac{\ell\varepsilon}{k} \right)\right)^{-N} \lesssim 1 .
\end{equation}
The remaining inequality~\eqref{eqn:rectbound} will be a consequence of the
following lemma.

\begin{lemma}\label{lem:geomrectangle}
  Let $R$ be a polar rectangle of the form
  \begin{equation*}
    R = \left\{(r,\theta) : r_0 < r \leq r_0+\varepsilon, \
      |\theta-\theta_0| \leq \frac{\varepsilon}{r_0}, \text{ and } |\theta| \leq
      \pi \right\}
  \end{equation*}
  lying either in the upper or lower half plane in $\RR^2$, where
  $\varepsilon>0$ is a sufficiently small parameter.  Then for any $r_1>0$,
  \begin{equation}\label{eqn:unitbound}
    \iint_R \left(t^2-r^2_1-r^2 +2r_1r \cos(\theta) \right)^{-\frac{1}{2}}_+
    \; r \, dr \, d\theta \lesssim \min\!\left(t,t^{-\frac{1}{2}}\right).
  \end{equation}
\end{lemma}
We postpone the proof of the lemma and show how it yields the estimate
\eqref{eqn:rectbound}.  Observe that for any rectangle defined above, either
\begin{equation*}
  R_{k,\ell} \subset (0,\infty) \times [0,\min(\pi,\pi\rho)]
  \quad \text{or}\quad
  R_{k,\ell} \subset
  (0,\infty) \times (-\min(\pi,\pi\rho),0].
\end{equation*}
We now consider the natural identification of $(0,\infty) \times
[0,\min(\pi,\pi\rho)]$ and $(0,\infty) \times (-\min(\pi,\pi\rho),0]$ with a
convex subset of either the upper half plane or the lower half plane.  When
$j=0$ in~\eqref{eqn:psidefn}, the desired estimate on the corresponding term
follows directly from the lemma.  Otherwise, we make a change of coordinates
$\tilde{\theta} = \theta_2 + j \cdot 2\pi\rho$ and the condition $|\theta_2 + j
\cdot 2\pi\rho| \leq \pi$ means that it is sufficient to assume that
$|\tilde{\theta}|\leq \pi$ for any $(r,\tilde{\theta})$ in the translated
rectangle $R_{k,\ell}+(0,j\cdot2\pi\rho)$.  Hence the lemma yields the desired
estimate for other values of $j$.

\begin{proof}[Proof of Lemma~\ref{lem:geomrectangle}]
  It suffices to treat the case where $0 \leq \theta \leq \pi$, as a slight
  adjustment of the arguments below will handle the remaining cases.  We switch
  to polar coordinates $(\tilde{r},\tilde{\theta})$ centered at $(r_1,0)$, and
  let $d_0$ be the fixed distance
  \begin{equation*}
    d_0 = d_\g(r_1,0;r_0,\theta_0).
  \end{equation*}
  By the law of cosines, $r^2_1+r^2 -2r_1r \cos(\theta) =
  d^2_\g\!\left(r_1,0;r,\theta\right)$ is the square of the Euclidean distance
  from $(r_1,0)$ to $(r,\theta)$, and by~\eqref{eqn:diameter}, the diameter of
  $R$ is $\mathrm{O}(\varepsilon)$.  Given these observations, there exists a
  uniform constant $C>0$ and an angle $\tilde{\theta}_0$ such that $R \subset
  \widetilde{R}$ where
  \begin{equation*}
    \widetilde{R} \defeq \left\{(\tilde{r},\tilde{\theta}) : d_0-C\varepsilon
      \leq \tilde{r} 
      \leq d_0+C\varepsilon \, \text{ and }
      \left|\tilde{\theta}-\tilde{\theta}_0 
      \right| \leq \frac{C \, \varepsilon}{d_0} \right\}.
  \end{equation*}
  Since the change of coordinates is an isometry, we have the following bounds
  on the integral in~\eqref{eqn:unitbound}:
  \begin{equation}\label{eqn:tildechange}
    \iint_R \left[t^2-d^2_\g(r_1,0;r,\theta)\right]^{-\frac{1}{2}}_+ \;
    r\,dr\,d\theta \leq \iint_{\widetilde{R}}
    \left(t^2-\tilde{r}^2\right)^{-\frac{1}{2}}_+
    \; \tilde{r}\,d\tilde{r}\,d\theta.
  \end{equation}
  When $t \geq 1$, we use that $(t+\tilde{r})^{-\frac 12} \leq t^{-\frac 12}$
  and bound the integral on the right by
  \begin{equation*}
    t^{-\frac{1}{2}}\left(d_0 + C\varepsilon\right) \int_{r_0}^{r_0
      +\varepsilon} \int_{\left|\tilde{\theta}-\tilde{\theta}_0\right| \leq
      \frac{C \, \varepsilon}{d_0}} \left(t-\tilde{r}\right)^{-\frac{1}{2}}_+
    \; d\tilde\theta\,d\tilde{r }\lesssim t^{-\frac{1}{2}}(d_0 +
    C\varepsilon) \cdot \frac{C \, \varepsilon}{d_0} \lesssim t^{-\frac{1}{2}}.
  \end{equation*}
  An straightforward adjustment of this computation handles the case where $d_0
  \lesssim \veps$. When $t \leq 1$, we can take the limits of integration in $r$
  so that they do not exceed $t$.  Hence, the integral on the left
  in~\eqref{eqn:tildechange} can be bounded by a constant multiple of
  \begin{equation*}
    t\int_{0}^\pi \int_{0}^t
    \left(t^2-\tilde{r}^2\right)^{-\frac{1}{2}}_+
    \;d\tilde{r}\,d\tilde\theta =\frac{\pi^2 t}{2}.
  \end{equation*}
  This concludes the proof.
\end{proof}

\subsection{Bounds on the ``diffractive'' term}
\label{sec:diffbounds}
In order to handle the diffractive terms, we prove the following:
\begin{lemma}\label{lem:diffpointwise}
  The kernel $K_{\calU(t)}^\diff$ defined in~\eqref{eqn:diffractivesine}
  satisfies the following pointwise bounds with an implicit constant independent
  of $\theta_1$ and $\theta_2$:
  \begin{equation}\label{eqn:ptwisemain}
    \left|K_{\calU(t)}^\diff(t, r_1, \theta_1;r_2, \theta_2)\right| \lesssim
    \left[t^2-(r_1+r_2)^2\right]^{-\frac{1}{2}}_+.
  \end{equation}
\end{lemma}

Given this lemma, a straightforward adaptation of the proof of
Lemma~\ref{lem:geomrectangle} shows that
\begin{equation*}
  \iint_{R_{k,\ell}}\left|K_{\calU(t)}^\diff(t, r_1, \theta_1;r_2,
    \theta_2)\right| 
  \, r_1 \, dr_1 \, d\theta_1 \lesssim \min\!\left(t,t^{-\frac{1}{2}} \right) ; 
\end{equation*}
indeed, the proof of this fact is simpler since coordinates need not be shifted
away from the cone tip.  The desired dispersive estimate will then follow from
the reasoning preceding Lemma \ref{lem:geomrectangle}. The only difference here
is that the tiling of polar rectangles must now occur over $(0,\infty)_r \times
(-\pi \rho, \pi \rho]_\theta$.

\begin{proof}[Proof of Lemma~\ref{lem:diffpointwise}]
  Recall the notation
  \begin{alignat*}{3}
    \alpha &\defeq \frac{t^2-r_1^2-r_2^2}{2r_1 r_2} =
    \frac{t^2-(r_1+r_2)^2}{2r_1 r_2} + 1 &\hspace*{1.5cm}
    \beta &\defeq \cosh^{-1}(\alpha) \\
    \varphi_1 &\defeq \frac{\pi + \left(\theta_1 -\theta_2\right)}{\rho} &
    \varphi_2 &\defeq \frac{\pi - \left(\theta_1 -\theta_2\right)}{\rho}.
  \end{alignat*}
  Given the factor $\left(2r_1 r_2\right)^{-\frac{1}{2}}$
  in~\eqref{eqn:diffractivesine}, it suffices to show for $\varphi =
  \varphi_1,\varphi_2$ that
  \begin{equation}\label{eqn:alphabound}
    \left|\int_0^\beta \left[ \alpha - \cosh(s) \right]^{-\frac{1}{2}} \left[
        \frac{\sin(\varphi)}{\cosh(s/\rho)-\cos(\varphi)} \right]ds\right|
    \lesssim(\alpha-1)^{-\frac 12}.
  \end{equation}
  We will treat the cases $\beta \geq 1$ and $\beta \leq 1$ separately.  In the
  first case, we will approximate $\beta = \cosh^{-1}(\alpha) \approx
  \log(2\alpha)$.  In the second case, we will approximate $\beta \approx
  \sqrt{2(\alpha-1)}$.

  We begin with the case where $\beta \geq 1$, which in turn implies that
  $\alpha \geq \cosh(1) > \frac{3}{2}$.  We will choose a constant $M = M(\beta)
  \in \left[\frac{1}{2},\beta\right)$ and write the domain of integration
  in~\eqref{eqn:alphabound} as $\left[0,\beta\right) = [0,M) \sqcup [M,\beta)$.
  The constant $M$ will be chosen so that the following three conditions are
  satisfied.  First, we want a small uniform constant $\delta>0$ so that for $s
  \in [0,M]$,
  \begin{equation}\label{eqn:smalls}
    \cosh(s) - 1 \leq (1-\delta) (\alpha -1) =
    (1-\delta)\left(\frac{t^2-(r_1+r_2)^2}{2r_1 r_2}\right).
  \end{equation}
  For $s \in [M, \beta)$, we linearize $\cosh(s)$:
  \begin{equation}\label{eqn:linearize}
    \cosh(s) = \cosh(\beta) + (s-\beta) \sinh(\beta) + R_1(s,\beta).
  \end{equation}
  The second condition is to ensure that $M$ is sufficiently large so that for
  $s \in [M,\beta)$,
  \begin{equation}\label{eqn:remainder}
    \left|R_1(s,\beta)\right| \leq \frac{1}{2} (\beta-s)\sinh(\beta).
  \end{equation}
  Finally, we want to see that $M$ is bounded from below:
  \begin{equation}\label{eqn:boundedbelow}
    M \geq \max\!\left(\beta -1,\frac{1}{2}\right).
  \end{equation}

  Assuming that this choice of $M$ can be made, we now turn to the integral
  in~\eqref{eqn:alphabound}. We first treat the integral over $[0,M]$.
  Given~\eqref{eqn:smalls}, we have that for small $\varphi$
  (cp.~\cite[(4.15)]{CT82no2})
  \begin{equation}\label{eqn:arctan}
    \int_0^M \left[ \alpha - \cosh(s) \right]^{-\frac{1}{2}}\left[ \frac{|\sin
        (\varphi)|}{\cosh(s/\rho) - \cos(\varphi)} \right]\,ds \lesssim
    (\alpha-1)^{-\frac 12} \int_0^M \frac{|\varphi|}{\varphi^2+s^2}\;ds.
  \end{equation}
  A similar argument works for $\varphi$ in neighborhoods of integral multiples
  of $2\pi$.  Otherwise, stronger estimates hold since the denominator in
  brackets is bounded away from 0.  Hence, this term is bounded by
  $(\alpha-1)^{-\frac{1}{2}}$.  Turning to the integral over $[M,\beta)$, the
  approximations \eqref{eqn:linearize} and \eqref{eqn:remainder} imply that
  \begin{equation*}
    \alpha - \cosh(s) = (\beta-s) \sinh(\beta) - R_1(s,\beta) \geq \frac 12
    (\beta-s)\sinh(\beta).
  \end{equation*}
  Furthermore, choosing $M \geq \frac 12$ ensures that the denominator in
  brackets can be uniformly bounded from below when $s \in [M, \beta]$.  Thus,
  for any $\varphi$ we see
  \begin{equation}\label{eqn:sinhintegral}
    \int_M^\beta \left[ \alpha - \cosh(s) \right]^{-\frac 12}\left[ \frac{|\sin 
        (\varphi)|}{\cosh(s/\rho)-\cos(\varphi)} \right]\;ds \lesssim
    \int_M^\beta \sinh(\beta)^{-\frac 12} (\beta-s)^{-\frac 12}\;ds.
  \end{equation}
  The inequality $M \geq \beta-1$ in~\eqref{eqn:boundedbelow} ensures that the
  integral on the right is bounded by a constant multiple of
  \begin{equation*}
    \sinh(\beta)^{-\frac 12} = \left[\cosh(\beta) - e^{-\beta}
    \right]^{-\frac{1}{2}} \lesssim (\alpha-1)^{-\frac 12},
  \end{equation*}
  with the inequality following from the fact that $\beta \geq 1$.

  We now show that choosing $ M= \max\!\left(\beta - \tanh(\beta),
    \frac{1}{2}\right)$ satisfies the requisite
  estimates~\eqref{eqn:smalls},~\eqref{eqn:remainder},
  and~\eqref{eqn:boundedbelow}.  The last property follows easily from the fact
  that $\tanh(\beta) \leq 1$.  The remainder estimate~\eqref{eqn:remainder} also
  follows easily, for any $s \in [M, \beta]$ will satisfy $ \beta-s \leq
  \tanh(\beta)$, and hence
  \begin{equation*}
    \left|R_1(s,\beta)\right| \leq \frac 12(\beta-s)^2\cosh(\beta) \leq \frac
    12(\beta-s)\cosh(\beta) \tanh(\beta) \leq \frac 12(\beta-s)\sinh(\beta).
  \end{equation*}
  All that remains is to check that with this choice of $M$ there exists
  $\delta>0$ so that~\eqref{eqn:smalls} is satisfied.  Observe that for $\beta
  \geq 1$, $e^{\beta} < 2\cosh(\beta) =2\alpha \leq 2e^{\beta}$. Therefore, if
  $M= \beta-\tanh(\beta)$ and $s \in [0, M]$, then
  \begin{equation*}
    \begin{aligned}
      \cosh(s) -1 &\leq \exp\!\left[\beta-\tanh(\beta)\right]-1 \\
      &\leq 2\alpha\exp\!\left[-\tanh(1)\right] - 1 \\
      & \leq \alpha \exp\!\left[\log(2) - \tanh(1)\right] - 1.
    \end{aligned}
  \end{equation*}
  Since $\log(2) - \tanh(1) < 0$,~\eqref{eqn:smalls} will then hold for any
  $\delta>0$ satisfying $1-\delta > \exp\!\left[\log(2) - \tanh(1) \right]$.  If
  $M=\frac 12$, then since $\cosh(1) \leq \cosh(\beta) = \alpha$, we have
  \begin{equation*}
    \cosh\!\left(\frac 12\right) - 1 \leq \frac{\cosh\!\left(\frac
        12\right)}{\cosh(1)} \alpha -1.
  \end{equation*}
  Similarly, \eqref{eqn:smalls} will hold provided $1-\delta >
  \frac{\cosh(1/2)}{\cosh(1)}$.

  We now turn to the case $\beta \leq 1$ (equivalently $\alpha \leq \cosh(1)$).
  Using the inverse function theorem, we observe $\beta^2$ is differentiable in
  $\alpha$ when $\alpha>1$, and
  \begin{equation*}
    \frac{d }{d\alpha}\beta^2 = 2\beta\frac{d\beta}{d\alpha} =\frac{2 \beta}{\sinh
      (\beta)}.
  \end{equation*}
  By l'H\^opital's rule, this function is bounded from below as $\beta \to 0$.
  Furthermore, $\frac{2 \beta}{\sinh(\beta)}$ is a decreasing function for
  $\beta \geq 0$.  The fundamental theorem of calculus thus allows us to
  conclude that for any $\alpha \leq \cosh(1)$, we have $\beta^2 \geq \frac
  2{\sinh(1)}(\alpha-1) > 1.7(\alpha-1)$.  Since $\cosh(\beta) \geq 1+
  \frac{\beta^2}{2}$, we have the upper and lower bounds
  \begin{equation}\label{eqn:sqrtapprox}
    \sqrt{2(\alpha-1)} \geq \beta \geq \sqrt{1.7(\alpha-1)}.
  \end{equation}

  As before, we will select $M=M(\beta) \in (0,\beta)$ and estimate the integral
  in~\eqref{eqn:diffractivesine} by partitioning $\int_0^\beta = \int_0^M +
  \int_M^\beta$.  We begin by insisting that for $s \in [0,M]$,
  \begin{equation}\label{eqn:smallbeta-s}
    \cosh(s) -1 \leq \frac 12 (\alpha -1).
  \end{equation}
  We also require that for $s \in [M,\beta]$, the remainder
  estimate~\eqref{eqn:remainder} in the linear approximation is satisfied.
  Finally, we insist that $M \geq \frac {\beta}{2}$.  If these conditions hold,
  then \eqref{eqn:smallbeta-s} ensures that the integral over $[0,M]$ can be
  estimated as in~\eqref{eqn:arctan}.  However, this time we need a different
  argument to estimate the integral on the left in~\eqref{eqn:sinhintegral}.
  For small $\varphi$, we have
  \begin{multline}\label{eqn:smallbeta}
    \int_M^\beta \left[ \alpha - \cosh(s) \right]^{-\frac 12} \left[ \frac{|\sin
        (\varphi)|}{\cosh(s/\rho) - \cos(\varphi)} \right] ds\\
    \lesssim \int_M^\beta \left[(\beta - s) \sinh(\beta) \right]^{-\frac
      12}\left[ \frac{|\sin(\varphi)|}{\cosh(M/\rho) - \cos(\varphi)}
    \right] ds\\
    \lesssim \frac{|\varphi|}{\varphi^2+\beta^2} \sinh(\beta)^{-\frac
      12}\int_{\frac{\beta}{2}}^\beta (\beta-s)^{-\frac 12}\;ds \lesssim
    \frac{1}{\beta}\left(\frac{\beta}{\sinh(\beta)}\right)^{\frac 12}.
  \end{multline}
  We can now conclude that
  \begin{equation*}
    \frac{1}{\beta} \left(\frac{\beta}{\sinh(\beta)}\right)^{\frac 12} \lesssim
    (\alpha-1)^{-\frac 12}
  \end{equation*}
  by~\eqref{eqn:sqrtapprox}.

  This time, we select $M=\sqrt{\frac 12 (\alpha-1)}$ and observe that $M \geq
  \frac{\beta}{2}$ by~\eqref{eqn:sqrtapprox}.  Note that this choice of $M$
  implies~\eqref{eqn:smallbeta-s} is satisfied.  Indeed, by Taylor remainder
  estimates we have that when $|s|\leq 1$,
  \begin{equation*}
    \left|\cosh(s) -1 -\frac {s^2}2\right| \leq s^4 \, \frac{\cosh
      (1)}{4!} ,
  \end{equation*}
  which in turn implies that for $s \in [0,M]$
  \begin{equation*}
    \left|\cosh(s) -1 \right| \leq s^2 \leq \frac{1}{2} (\alpha -1).
  \end{equation*}

  To see that the remainder estimate~\eqref{eqn:remainder} is satisfied here for
  $s \in [M,\beta]$, we will show that
  \begin{equation*}
    \beta - \tanh(\beta) \leq \frac{\beta}{2} \leq \sqrt{\frac 12 (\alpha-1)}= M.
  \end{equation*}
  The second inequality follows from~\eqref{eqn:sqrtapprox}, so it suffices to
  prove the first.  This is equivalent to showing that
  \begin{equation*}
    \frac 12 \beta - \tanh(\beta) \leq 0.
  \end{equation*}
  As a function of $\beta \geq 0$, $\frac{1}{2} \beta - \tanh(\beta)$ is convex
  and decreasing near 0.  Therefore, the desired inequality holds for $\beta \in
  [0,1]$ since $\tanh(1) > \frac{1}{2}$.
\end{proof}

\section{Applications}
\label{sec:apps}

In this section, we present some applications of the above analysis.  We begin
by discussing Morawetz estimates on Euclidean cones and, more generally, metric
cones.  We then present Strichartz estimates for the appropriate wave equation
IVPs on wedge domains, polygonal domains, and Euclidean surfaces with conic
singularities.  We close by discussing well-posedness results for nonlinear wave
equations in these settings.

\subsection{Morawetz Estimates on the Euclidean Cone}
\label{sec:morawetz}

Following ideas in \cite{BPST-Z} and \cite{PST-Z}, we prove Morawetz, or local
energy decay, estimates for the wave equation on Euclidean cones.  To begin, we
define the multiplier operator
\begin{eqnarray*}
  (\Omega^s \phi) (r, \theta) = r^s \phi (r,\theta).
\end{eqnarray*}

Recall from Section \ref{sec:sobol-spac} that our definition of the homogeneous
Sobolev spaces $\dot{H}^s\!\left(\cone\right)$ is in terms of the spectral
decomposition of the Laplace-Beltrami operator $\Delta_\g$.  We may thus define
the following subspace,
\begin{equation}
  \dot{H}_{\geq m}^{s}\!\left(\cone\right) \defeq \left\{ f \in \dot{H}^s\!\left
      (\cone\right): \Pi_j f = 0 \text{ for } \nu_j < \nu_m \right\},
\end{equation}
where $\Pi_j$ is the spectral projector defined in \eqref{eq:4} and $\nu_m$ is
the $m$-th modified eigenvalue of $\Delta_\g$.  We have the following theorem.
\begin{theorem}
  \label{thm:morawetz}
  Let $m \geq 1$ be an integer and $0<\alpha<\frac14 + \frac 12\nu_m$.  Given a
  solution $u$ to the wave equation IVP \eqref{eqn:wave}, there exists a
  constant $C = C(m,\alpha)$ such that for all $f \in \dot{H}_{\geq m}^{\frac12}
  \!\left(\cone\right)$ and $g \in \dot{H}_{\geq
    m}^{-\frac12}\!\left(\cone\right)$, we have
  \begin{equation}\label{eqn:morawetz}
    \left\| \Omega^{-\frac12 - 2 \alpha} \Delta_\g^{\frac14 - \alpha} u
    \right\|_{L^2 ( \RR \times \cone)} \leq C \left( \| f \|_{ \dot{H}_{\geq
          m}^{\frac12} (\cone)} + \| g \|_{ \dot{H}_{\geq m}^{-\frac12}
        (\cone)} \right). 
  \end{equation}
\end{theorem}

\begin{remark}
  The proof may be performed on each spherical harmonic separately; this follows
  directly from the Hankel function analysis in \cite{BPST-Z}.  We note that the
  proof involves only analysis on the radial operator, so the Morawetz estimate
  will hold in general for any metric cone $C(Y)$ as described in Section
  \ref{sec:cheeg-funct-calc}.  In particular, we note that $\nu_0 > 0$ when
  $\dim(Y) > 1$, removing any restrictions on the bottom of the spectrum
\end{remark}

The proof of Theorem \ref{thm:morawetz} relies heavily on the results in
\cite{PST-Z} and \cite{BPST-Z}. By using the projections $\Pi_j$ defined in
\eqref{eq:4} and orthogonality, it suffices to consider $u$, $f$, and $g$ in the
range of $\Pi_j$ for some $j$.  As a consequence, we may assume that all three
functions depend only on $t$ and $r$ and that
\begin{equation*}
  \left\{
    \begin{aligned}
      \left(D_t^2 - A_\nu \right) u(t,r) &= 0 \\
      u(0,r) &= f(r) \\
      \prtl_t u (0,r) &= g(r) , \\
    \end{aligned}
  \right.
\end{equation*}
where $\nu^2 = \nu_j^2$, the $j$-th (modified) eigenvalue of
$\Delta_{\bbS^1_\rho}$, and
\begin{equation*}
  A_\nu \defeq - \p_r^2 - \frac{1}{r} \p_r + \frac{\nu^2 }{r^{2}}.
\end{equation*}
We may then define powers of $A_\nu$ by
\begin{equation*}
  A^{\frac{\sigma}{2}}_\nu \defeq \mathcal{H}_\nu \circ \Omega^\sigma \circ
  \mathcal{H}_\nu. 
\end{equation*}

In \cite{BPST-Z}, it is shown that the desired estimate \eqref{eqn:morawetz}
reduces to showing
\begin{equation}\label{eqn:morawetzproj}
  \left\| \Omega^{-\frac12 - 2 \alpha} A_\nu^{\frac14 - \alpha} u \right\|_{L^2
    ( \RR \times \cone)} \lesssim  \left\| A_\nu^{\frac 14} f \right\|_{
    L^2(\cone)} + \left\| A_\nu^{-\frac 14} g \right\|_{L^2(\cone)}. 
\end{equation}
Observe that
\begin{equation*}
  \left(\mathcal{H}_\nu u\right)\!(t,\lambda) = \cos( t \lambda) (
  \mathcal{H}_\nu f )( \lambda) + 
  \frac{\sin (t \lambda)}{ \lambda} (\mathcal{H}_\nu g ) ( \lambda ).
\end{equation*}
Taking the Fourier transform of $u$ in time yields
\begin{equation*}
  \left(\mathcal{F}_t \mathcal{H}_\nu u\right)\!(\tau,\lambda) =
  \frac{1}{\sqrt{\lambda}} 
  \big(\delta (\tau + \lambda) h_+ (\lambda) + \delta ( \tau - \lambda) h_{-}
  (\lambda)\big), 
\end{equation*}
where
\begin{equation*}
  h_{\pm} (\lambda) = \frac12 \left(\sqrt{\lambda} \left(\mathcal{H}_\nu
      f\right)\!( \lambda) 
    \pm  \frac{1}{i \sqrt{\lambda}} \left(\mathcal{H}_\nu g\right)\!( \lambda)
  \right). 
\end{equation*}
Moreover,
\begin{equation*}
  \left\|h_{\pm}\right\|_{L^2(\RR_+;\l d\l)} \lesssim \left\| A_\nu^{\frac 14} f
  \right\|_{L^2(\cone)} 
  + \left\| A_\nu^{-\frac 14} g \right\|_{L^2(\cone)}.
\end{equation*}

For $\tau > 0$,
\begin{equation*}
  \left( A_\nu^{-\frac14 - \alpha} \Omega^{\frac12 - 2 \alpha} \mathcal{F}_t
    \mathcal{H}_\nu u \right)\!(\tau, \lambda) = \tau^{1-2\alpha} \,
  k_{\nu,\nu}( \lambda, \tau)^{-\frac12 - 2 \alpha} \, h_{-} (\tau ),
\end{equation*}
where $k_{\nu,\nu}$ is the integral kernel of $A_{\nu}^{\frac{\sigma}{2}}$ as in
\cite{BPST-Z} and related to that in \eqref{eq:14}.  Correspondingly, there is
an expression for $\tau < 0$ involving $h_+$.  The inequality
\eqref{eqn:morawetzproj} then follows from the calculation in \cite{BPST-Z}
which establishes that
\begin{equation*}
  \left\| A_\nu^{-\frac14 - \alpha} \Omega^{\frac12 - 2 \alpha} \mathcal{F}_t
    \mathcal{H}_\nu u \right\|_{L^2 (\RR_\pm \times \RR_+; \l d\tau d\l )} =
  C_{\nu,\alpha}\|h_\mp\|_{L^2(\RR_+; \l d\l)}
\end{equation*}
for some constant $C_{\nu,\alpha}$ uniformly bounded in $\nu$.

\subsection{Global Strichartz estimates on wedge domains}
\label{sec:wedges}
Let $\Omega$ be a planar domain of the form
\begin{equation*}
  \Omega = \left\{ (r,\theta) : r>0 \text{ and } 0 \leq \theta \leq \alpha
  \right\} \subseteq \bbR^2 ,
\end{equation*}
where $(r,\theta)$ denote standard polar coordinates and $\alpha \leq 2\pi$.  We
consider solutions to the inhomogeneous IBVP for the wave equation,
\begin{equation}
  \label{eqn:omegawave}
  \left\{
    \begin{aligned}
      \left(D_t^2 - \Delta \right) u(t,r,\theta) &= F(t,r,\theta) \\
      u(0,r,\theta) &= f(r,\theta) \\
      \prtl_t u (0,r,\theta) &= g(r,\theta) , \\
    \end{aligned}
  \right.
\end{equation}
satisfying either Dirichlet or Neumann homogeneous boundary conditions, i.e.
\begin{equation}\label{eqn:bdrycond}
  u\big|_{\RR \times \prtl \Omega} \equiv 0 \qquad \text{ or } \qquad \prtl_n
  u\big|_{\RR \times \prtl \Omega} \equiv 0  . 
\end{equation}
Here, $\prtl_n$ denotes the normal derivative along the boundary.  When
Dirichlet conditions are imposed, we take $\Delta$ to mean the Friedrichs
extension of the Laplacian $D_{x_1}^2 + D_{x_2}^2$ acting on
$\mathcal{C}^\infty_c(\Omega)$.  If Neumann conditions are imposed, we take
$\Delta$ to mean the Friedrichs extension of the same differential operator
acting on smooth functions which vanish in a neighborhood of the vertices and
whose normal derivative vanishes on $\prtl \Omega$.

In this section, we discuss global Strichartz on these domains.  Namely, we
prove the following theorem.
\begin{theorem}\label{thm:globalstz}
  Suppose $u$ is a solution to the wave equation IBVP
  ~\eqref{eqn:omegawave}-\eqref{eqn:bdrycond} with initial data $(f,g) \in
  \dot{H}^\gamma(\Omega) \times \dot{H}^{\gamma-1}(\Omega)$.  Then for any
  triple $(p,q,\gamma)$ satisfying the hypotheses of Theorem~\ref{thm:strich},
  $u$ satisfies the estimates
  \begin{multline}\label{eqn:wedgestrich}
    \| u \|_{L^p(\RR;L^q(\Omega))} + \| (u,\prtl_t u)
    \|_{L^\infty(\RR;\dot{H}^\gamma(\Omega) \times \dot{H}^{\gamma-1}(\Omega))
    }\\ \lesssim \|(f,g)\|_{\dot{H}^\gamma(\Omega)\times
      \dot{H}^{\gamma-1}(\Omega)} + \| F \|_{L^{\tilde{p}'}(\RR;
      L^{\tilde{q}'}(\Omega))},
  \end{multline}
\end{theorem}

Recently, there has been some partial progress in proving such bounds for
domains with smooth boundaries; see~\cite{SmSoCrit},~\cite{BLP},
and~\cite{BSSpoin}.  However, the examples of Ivanovici \cite{iv1} show that
when the smooth boundary possesses a point of convexity, the full range of
scale-invariant Strichartz estimates cannot hold.  In contrast, we will obtain
the full range of Strichartz estimates, even for convex domains.  The fact that
the boundaries we consider are flat, except at the corners, seems to limit the
problems created by multiply-reflected rays.  However, as is implicit in Section
\ref{sec:ests}, the diffractive effects created by interaction with the corners
create their own challenges.

We note that the diffraction occurring from the tip of a wedge was first
analyzed by Sommerfeld in \cite{Som} via special function computations.  For a
nice treatment, we recommend the book by Stakgold \cite[Chapter 7.12]{Stak}, who
uses the method of reflection to describe Green's function for the Helmholtz
operator on a wedge.

\begin{proof}[Proof of Theorem~\ref{thm:globalstz}]
  By density arguments similar to those in Section \ref{sec:sobol-spac}, it
  suffices to assume that $u(t,\cdot)$, $f$, and $g$ are all smooth and
  compactly supported in $\overline{\Omega}$.  Furthermore, by taking Fourier
  series in $\theta$, we may assume that $u(t,r,\theta)$ may be written as
  either
  \begin{equation}\label{eqn:fourierexpntime}
    \frac{1}{\sqrt{\alpha}}\sum_{j=1}^\infty u_j(t,r) \sin\!\left(\frac{j\pi
        \theta}{\alpha}\right) \quad \text{or} \quad u_0(t,r)
    +\frac{1}{\sqrt{\alpha}} \sum_{j=1}^\infty u_j(t,r) \cos\!\left(\frac{j \pi
        \theta}{\alpha}\right) ,
  \end{equation}
  depending on whether Dirichlet or Neumann boundary conditions are taken.  The
  separation of variables approach in Section~\ref{sec:Spaces} can now be
  applied to these expansions, giving rise to a spectral resolution of the
  Dirichlet or Neumann Laplacian.  We may then use this spectral resolution to
  define homogeneous Sobolev spaces on $\Omega$.

  Given \eqref{eqn:fourierexpntime}, $u(t,r,\theta)$ can be extended to a
  function on $\cone$ for $\rho =\frac{\alpha}{\pi}$.  In addition,
  $L^2\!\left(\cone\right)$ may be decomposed as a sum of the $L^2$-orthogonal
  subspaces $L^2_\Dir\!\left(\cone\right)$ and $L^2_\Neu\!\left(\cone\right)$.
  The subspace $L^2_\Dir\!\left(\cone\right)$ is defined as all functions $a \in
  L^2\!\left(\cone\right)$ which may be written
  \begin{equation}\label{eqn:fouriersine}
    a(r,\theta) = \frac{1}{\sqrt{\alpha}}\sum_{j=1}^\infty a_j(r)
    \sin\!\left(\frac{j\pi  \theta}{\alpha}\right). 
  \end{equation}
  Similarly, $L^2_\Neu\!\left(\cone\right)$ is defined as all functions which
  can be written as \begin{equation}\label{eqn:fouriercosine} a(r,\theta) =
    a_0(r) + \frac{1}{\sqrt{\alpha}}\sum_{j=1}^\infty a_j(r)
    \cos\!\left(\frac{j\pi \theta}{\alpha}\right).
  \end{equation}
  It is not difficult to see that $L^2_\Dir\!\left(\cone\right) \cap
  \Dom(\Delta_\g)$ and $L^2_\Neu\!\left(\cone\right) \cap \Dom(\Delta_\g)$ are
  invariant under the action of $\Delta_\g$.  Therefore, solutions $u(t,\cdot)$
  to the Dirichlet and Neumann problems extend naturally to functions in
  $L^2_\Dir\!\left(\cone\right)$ and $L^2_\Neu\!\left(\cone\right)$
  respectively.  Furthermore, there is a natural identification between the
  action of $\Delta^{\frac s2}$ on $L^2(\Omega)$ and the action of
  $\Delta_\g^{\frac s2}$ on $L^2_\Dir\!\left(\cone\right)$ and
  $L^2_\Neu\!\left(\cone\right)$.  As such,
  \begin{equation*}
    \|f\|_{\dot{H}^{\gamma}(\Omega)}\approx \|f\|_{\dot{H}^{\gamma}(\cone)},
  \end{equation*}
  and similarly for $g$.  Thus, by Theorem \ref{thm:strich} we have that
  \begin{equation}\label{eqn:extension}
    \| u \|_{L^p(\RR;L^q(\Omega))} \approx  \| u \|_{L^p([-T,T];L^q(\cone))}
    \lesssim \| f \|_{\dot{H}^\gamma(\cone)} +   \| g
    \|_{\dot{H}^{\gamma-1}(\cone)}.
  \end{equation}
  This concludes the proof.
\end{proof}

\subsection{Polygonal Domains and Euclidean Surfaces with Conic Singularities}
\label{sec:poly}

In this section, we let $\Omega$ be an open domain in the plane with a piecewise
linear boundary (i.e.~ a polygonal domain), and we consider local Strichartz
estimates for solutions to the IBVP \eqref{eqn:omegawave}-\eqref{eqn:bdrycond}.
We prove the following theorem.

\begin{theorem}\label{thm:localstz}
  Suppose $u$ is a solution to the wave equation IBVP
  ~\eqref{eqn:omegawave}-\eqref{eqn:bdrycond} with initial data $(f,g) \in
  H^\gamma(\Omega) \times H^{\gamma-1}(\Omega)$.  Then for any triple
  $(p,q,\gamma)$ satisfying the hypotheses of Theorem~\ref{thm:strich}, $u$
  satisfies the estimates
  \begin{multline}\label{eqn:polystrich}
    \| u \|_{L^p([-T,T];L^q(\Omega))} + \| (u,\prtl_t u)
    \|_{L^\infty([-T,T];H^\gamma(\Omega) \times H^{\gamma-1}(\Omega)) }\\
    \lesssim \|(f,g)\|_{H^\gamma(\Omega)\times H^{\gamma-1}(\Omega)} + \| F
    \|_{L^{\tilde{p}'}([-T,T]; L^{\tilde{q}'}(\Omega))},
  \end{multline}
\end{theorem}

\begin{proof}
  We begin by observing that it suffices to show this for sufficiently small $T$
  and smooth initial data $(f,g)$.  By finite propagation speed, this in turn
  allows us to assume that the solution is compactly supported in a set that is
  either isometric to a subset of the upper half plane or isometric to a corner
  of angle $\alpha$.  The most difficult case here is the latter; the former
  case will follow by Strichartz estimates on the plane and Sobolev space
  estimates analogous to those below.  The arguments here will also apply to any
  Euclidean surface with conic singularities, as seen in the Schr\"odinger
  equation case in \cite{BFHM}.

  We now assume that $f$ and $g$ are supported in a set of the form
  $\{(r,\theta):0 < r<\delta \text{ and } 0 \leq \theta \leq \alpha\}$.  As
  before, we expand $u(t,r,\theta)$ in the form~\eqref{eqn:fourierexpntime} and
  observe that $u$ extends to a function on the cone $\cone$ with $\rho=\frac
  \alpha\pi$.  For each time $t$, $u(t,\cdot)$ lies in one of the spaces
  $L^2_\Dir\!\left(\cone\right)$ or $L^2_\Neu\!\left(\cone\right)$ defined
  above.  Reasoning analogously to \eqref{eqn:extension},
  Theorem~\ref{thm:localstz} is a consequence of Corollary \ref{thm:localcorr}
  and the following lemma.
\end{proof}

\begin{lemma}\label{thm:sobolevapprox}
  Let $a(r,\theta)$ be a smooth function on a wedge of the form
  $\{(r,\theta):\delta_0 < r< \delta \text{ and } 0 \leq \theta \leq \alpha\}$
  that can be written in one of the forms \eqref{eqn:fouriersine} or
  \eqref{eqn:fouriercosine}.  Then for any $s \in [-2,2]$, there exists a
  constant $C$ independent of $\delta_0$ such that
  \begin{equation}\label{eqn:sobolevapprox}
    \| a \|_{H^s(\cone)} \leq C\| a\|_{H^s(\Omega)},
  \end{equation}
  where $\| a \|_{H^s(\cone)} $ denotes the quantity formed by extending the
  function to $\cone$, $\rho = \frac{\alpha}{\pi}$, and taking the corresponding
  Sobolev norm.
\end{lemma}

\begin{proof}
  When $s=2$, we have that
  \begin{equation*}
    \| f \|_{H^2(\cone)} \approx \sum_{k=0}^2 \left\|\nabla_\g^k
      f\right\|_{L^2(\cone)} \approx \sum_{k=0}^2 \left\|\nabla_\g^k
      f\right\|_{L^2(\Omega)} \approx \| f \|_{H^2(\Omega)}.
  \end{equation*}
  By interpolation with $s=0$,~\eqref{eqn:sobolevapprox} holds for any $s \in
  [0,2]$.  To handle the negative Sobolev indices, we let $\Pi_\Dir$ and
  $\Pi_\Neu$ denote the corresponding orthogonal projections onto
  $L^2_\Dir\!\left(\cone\right)$ and $L^2_\Neu\!\left(\cone\right)$
  respectively.  Now let $\chi=\chi(r)$ be a compactly supported smooth function
  on $\cone$ satisfying $\chi \equiv 1$ for $0\leq r\leq \delta$.  We then have
  that if $a$ can be expanded in terms of the sine basis, then its extension to
  the cone satisfies $\Pi_\Dir a =a$, and hence
  \begin{align*}
    \|a\|_{H^{-2}(\cone)} &= \sup \{|\langle a,b\rangle_{L^2(\cone)}| :
    \|b\|_{H^2(\cone)} =1 \}\\
    &= 2 \sup \{|\langle a, \chi \Pi_D b\rangle_{L^2(\Omega)}| :
    \|b\|_{H^2(\cone)} =1 \}\\
    &\lesssim\|a\|_{H^{-2}(\Omega)} \|\chi \Pi_D b\|_{H^{2}(\Omega)}
    \lesssim\|a\|_{H^{-2}(\Omega)}.
  \end{align*}
  Here, the last inequality follows from
  \begin{equation*}
    \|\chi \Pi_\Dir b\|_{H^{2}(\Omega)} \lesssim \|\chi \Pi_\Dir b\|_{H^{2}(\cone)}
    \lesssim \|\Pi_\Dir b\|_{H^{2}(\cone)} \leq \| b \|_{H^{2}(\cone)} .
  \end{equation*}
  When $a$ can be expanded in terms of the cosine basis, the same proof applies
  with $\Pi_\Dir$ replaced by $\Pi_\Neu$ to show~\eqref{eqn:sobolevapprox} when
  $s=-2$. Interpolation with $s=0$ now completes the proof.
\end{proof}

\subsection{Well-posedness for nonlinear waves}
\label{sec:nlw-wp}

As an application of the Strichartz estimates above, we note that these
inequalities can be used to show that a theorem of Lindblad and Sogge
\cite{LiSo} on $\RR^2$ carries over to any of the contexts above (wedge domains,
polygonal domains, Euclidean cones, and ESCSs).  Specifically, we let $X$ denote
any one of these manifolds, imposing Dirichlet or Neumann boundary conditions as
appropriate. Consider the semilinear initial value problem
\begin{equation}
  \label{eqn:semilinear-wave}
  \left\{
    \begin{aligned}
      \left(D_t^2 - \Delta_\g \right) v(t,x) &= \pm |v|^{\kappa - 1} v \\
      v(0,x) &= v_0(x) \in H^\gamma(X)\\
      \prtl_t v (0,x) &= v_1(x) \in H^{\gamma-1}(X) , \\
    \end{aligned}
  \right.
\end{equation}
where $\gamma=\gamma(\kappa)=\max\left(\frac 34 - \frac{1}{\kappa-1}, 1 -
  \frac{2}{\kappa-1}\right)$, i.e.\
\begin{equation*}
  \gamma=\gamma(\kappa) = \begin{cases} \frac 34 - \frac{1}{\kappa-1} & 3 <
    \kappa \leq 5\\ 
    1 - \frac{2}{\kappa-1} & 5 \leq \kappa < \infty .
  \end{cases}.
\end{equation*}

In general, the Sobolev index $\gamma(\kappa)$ is the lowest degree of
regularity for which the problem \eqref{eqn:semilinear-wave} is locally
well-posed.  Indeed, Lindblad-Sogge \cite{LiSo} constructed explicit examples of
solutions to the focusing problem on $\RR^2$ that demonstrate it is ill-posed in
Sobolev spaces below $\gamma(\kappa)$.  For defocusing nonlinearities,
Christ-Colliander-Tao \cite{CCT} produced examples on $\RR^2$ which show
ill-posedness when the Sobolev regularity is below $1 - \frac{2}{\kappa-1}$.
The well-posedness of the problem \eqref{eqn:semilinear-wave} is one of many
important applications of Strichartz estimates.  It is an example of how these
inequalities efficiently handle the perturbative theory for these equations.

\begin{theorem}
  Suppose $X$ is a 2-dimensional manifold where the local Strichartz estimates
  \eqref{eqn:polystrich} are valid.  Then given any pair of initial data in
  $(v_0,v_1) \in H^\gamma(X) \times H^{\gamma-1}(X)$ there exists $T>0$ and a
  unique solution of \eqref{eqn:semilinear-wave} satisfying
  \begin{equation*}
    \left(v(t,\cdot),\prtl_t v(t,\cdot)\right) \in \mathcal{C}^0\!\left( [0,T] ;
      H^\gamma(X) \times H^{\gamma-1}(X)\right) \cap L^{p}\!\left([0,T]; L^{\frac
        32(\kappa-1)}(X)\right) ,
  \end{equation*}
  where $p = \max\!\left(\frac{3}{\gamma(\kappa)}, \frac 32(\kappa-1)\right)$.
  Furthermore, if $T^*$ denotes the maximal lifespan of the solution in
  $H^\gamma(X) \times H^{\gamma-1}(X)$, then either $T^* = \infty$ or
  \begin{equation*}
    \|v\|_{L^{\frac 32(\kappa-1)}([0,T^*)\times X)} = \infty.
  \end{equation*}
\end{theorem}

When global Strichartz estimates are available, we also have global existence
when the initial data is sufficiently small in $\dot{H}^\gamma(X) \times
\dot{H}^{\gamma-1}(X)$.

\begin{corollary}
  Suppose $X$ is a 2-dimensional manifold where the global Strichartz estimates
  \eqref{eqn:wedgestrich} are valid.  Then the above theorem holds with
  inhomogeneous Sobolev spaces replaced by homogenous ones.  Moreover, there
  exists $\veps(\kappa) >0$ for which the solution will exist globally in
  $\dot{H}^\gamma(X) \times \dot{H}^{\gamma-1}(X)$ (i.e.~ $T^*=\infty$) whenever
  the initial data satisfies
  \begin{equation*}
    \|f\|_{\dot{H}^\gamma(X)} + \|g\|_{\dot{H}^{\gamma-1}(X)} \leq \veps(\kappa).
  \end{equation*}
\end{corollary}

The proofs of these results are essentially due to Lindblad-Sogge \cite{LiSo}.
The only complication is that the so-called ``fractional Leibniz rule'' is not
apparent in our context, so we must take additional care regarding the function
spaces we use to obtain a contraction.  In particular, we make use of the
estimate
\begin{multline}\label{eqn:inhomogstz}
  \| u \|_{L^{p}([0,T];L^{q}(X))} + \left\| ( u , \prtl_t u)
  \right\|_{L^\infty([0,T];H^\gamma(X) \times H^{\gamma-1}(X)) }\\
  \lesssim \left\| (f , g) \right\|_{H^\gamma(X)\times H^{\gamma-1}(X)} + \| F
  \|_{L^{\frac 3{2+\gamma}}\!\left([0,T];L^{\frac 6{7-4\gamma}}(X)\right)}.
\end{multline}
where $p$ and $q$ will satisfy \eqref{eqn:tripleadmiss} and $\frac 1p + \frac 2q
= 1-\gamma$.  When $3 < \kappa < 9$, we may take $p=\frac 3\gamma$ and $q=\frac
6{3-4\gamma}$ as in Lemmas 4.1 and 4.2 of \cite{LiSo} to obtain local existence
of the solution.  When $9 \leq \kappa < \infty$, we instead take $p=
\frac{3\kappa}{2+\gamma}$ and $q=\frac{6\kappa}{7-4\gamma}$, and local existence
follows from a slight adjustment of Lemma 3.6 in \cite{Sogge}.

The remainder of the theorem now follows by the same considerations as in
Theorems 5.1 and 5.2 in \cite{LiSo} once it is seen that $v \in L^{\frac
  32(\kappa-1)}([0,T]\times X)$.  When $k \geq 9$, H\"older's inequality gives
\begin{equation*}
  \|v\|_{L^{\frac 32 (\kappa-1)}([0,T]\times X)} \leq
  \|v\|_{L^{\frac{3\kappa}{2+\gamma}}\!
    \left([0,T];L^{\frac{6\kappa}{7-4\gamma}}(X)\right)}^{\theta}  
  \|v\|_{L^{\infty}\!\left([0,T];L^{\frac{2}{1-\gamma}}(X)\right)}^{1-\theta}
\end{equation*}
with $\theta = \frac{2\kappa}{(\kappa -1) (2+\gamma)}$.  The first factor on the
right is finite as it is the space used to perform the contraction.  Since
$H^\gamma(X)$ embeds into $L^{\frac{2}{1-\gamma}}(X)$, the second factor is also
finite.  A similar argument holds when $5 < \kappa \leq 9$; see
e.g. \cite[(4.17)]{LiSo}.  Finally, when $3 < \kappa \leq 5$, we have $v \in
L^{\frac 3\gamma}\!\left([0,T]; L^{\frac 32(\kappa-1)}(X)\right)$.  Since $\frac
3\gamma \geq \frac 32(\kappa-1)$ in this case, $v \in L^{\frac
  32(\kappa-1)}([0,T]\times X)$ follows by applying the H\"older inequality in
time.


\end{document}